\documentclass[12pt,psfig,reqno]{amsart}
\usepackage{mathrsfs}
\usepackage{txfonts}
\usepackage{cite}
\usepackage{amsmath}
\usepackage{bbm}
\usepackage[normalem]{ulem}
\usepackage{amssymb}
\usepackage{amsfonts}
\usepackage{amscd}
\usepackage{amssymb,amsfonts,latexsym}
\usepackage{graphics,verbatim}
\usepackage{graphicx}
\usepackage[usenames]{color} %\color{Red}
\usepackage{xcolor}
\makeatletter
\renewcommand\section{\@startsection {section}{1}{\z@}%
                                   {-3.5ex \@plus -1ex \@minus -.2ex}%
                                   {2.3ex \@plus.2ex}%
                                   {\centering\normalfont\bf}}

 \numberwithin{equation}{section}
\numberwithin{equation}{section}

\numberwithin{equation}{section}

\theoremstyle{plain}
\newtheorem{thm}{Theorem}[section]
\newtheorem{lemma}[thm]{Lemma}
\newtheorem{pro}[thm]{Proposition}
\newtheorem{cor}[thm]{Corollary}
\newtheorem{ex}[thm]{Example}
\newtheorem{de}[thm]{Definition}
\newtheorem{re}[thm]{Remark}

\newtheorem*{co*}{Conjecture}
\newtheorem*{thm*}{Theorem}

\begin{document}
\title{The spectrality of self-affine measure under the similarity transformation of  $GL_n(p)$}
\author{Jing-Cheng Liu, Zhi-Yong Wang$^*$}

\address{Key Laboratory of Computing and Stochastic Mathematics (Ministry of Education), School of Mathematics and Statistics, Hunan Normal University, Changsha, Hunan 410081, P. R. China}

\email{jcliu@hunnu.edu.cn}

\address{College of Mathematics and Computational Science, Hunan First Normal University, Changsha, Hunan 410205, P. R. China }

\email{wzyzzql@163.com}

\date{\today}
\thanks{ The research is supported in part by the NNSF of China (No. 11831007), the Hunan Provincial NSF (No. 2019JJ20012,  2020JJ5097), the SRF of Hunan Provincial Education Department (Nos. 17B158, 19B117).\\
$^*$Corresponding author.}

\begin{abstract}
Let $\mu_{M,D}$ be the self-affine measure generated by an expanding integer matrix $M\in M_n(\mathbb{Z})$ and  a finite digit set $D\subset\mathbb{Z}^n$. It is well  known that the two measures  $\mu_{M,D}$ and  $\mu_{\tilde{M},\tilde{D}}$ have the same spectrality if $\tilde{M}=B^{-1}MB$ and $\tilde{D}=B^{-1}D$, where $B\in M_n(\mathbb{R})$ is a nonsingular matrix. This fact is usually used to simplify the digit set $D$ or the expanding matrix $M$. However, it often transforms integer digit set $D$ or expanding matrix $M$ into real, which brings many difficulties to study the spectrality of $\mu_{\tilde{M},\tilde{D}}$.  In this paper, we introduce a similarity transformation of  general linear group  $GL_n(p)$ for some self-affine measures, and discuss their spectrality. This kind of similarity transformation can keep the integer properties of  $D$ and $M$ simultaneously, which leads to many advantages in discussing the spectrality of self-affine measures. As an application, we extend some well-known spectral self-affine measures to  more general forms.
\end{abstract}

\maketitle

\section{\bf Introduction\label{sect.1}}

Let $\mu$ be a Borel probability measure with compact support in $\mathbb{R}^n$ and let $\langle\cdot,\cdot\rangle$ denote the standard inner product on $\mathbb{R}^n$. We call $\mu$ a \emph{spectral
measure} if there exists a countable set $\Lambda$ such that
\begin{equation}\label {def:eq_E(A)}
E(\Lambda):=\{e^{2\pi i\langle\lambda,x\rangle} :  \lambda\in\Lambda\}
\end{equation}
forms an orthonormal basis
for $L^2(\mu)$. In this case, the set $\Lambda$ is called a \emph{spectrum} of $\mu$. We also say that a Borel subset $\Omega$ of $\Bbb R^n$ is a \emph{spectral set} if the normalized Lebesgue measure supported on $\Omega$ is a spectral measure. In fact, spectral measure is a generalization of the spectral set, which introduced by Fuglede \cite{Fuglede_1974} in his famous spectrum-tiling conjecture: \emph{$\Omega$ is a spectral set if and only if  $\Omega$ is a translational tile}.  In 2004, Tao \cite{Tao_2004} showed that ``spectral set implies tile'' is false. Subesequently, the conjecture was proved to be false in both directions on $\mathbb{R}^n$ for $n\geq3$ \cite{Kolountzakis-Matolcsi_2006-1, Kolountzakis-Matolcsi_2006-2}. Little is known about the situation in dimension one and two, and the conjecture remains open until now. However,  Greenfeld and Lev proved that Fuglede's conjecture is true for convex polytopes in $\mathbb{R}^3$\cite{Rachel-Lev_2017}. Especially,  Lev and Matolcsi get an amazing progress on the study of Fuglede's conjecture recently, they showed that Fuglede's conjecture for convex domains is true in all dimensions \cite{Lev-Matolcsi_2019}.

The study of spectral measure has entered the fractal field after the pioneer work of Jorgensen and Pedersen \cite{Jorgensen-Pedersen_1996, Jorgensen-Pedersen_1998}, who constructed the first example of singular, non-atomic spectral self-similar measure \cite{Jorgensen-Pedersen_1998}. Their construction is based on a scale $4-$Cantor set, where the first and third intervals are kept and the others are discarded. The results of Jorgensen and Pedersen \cite{Jorgensen-Pedersen_1998}
were further extended to other self-similar/self-affine/moran  fractal measures  (see e.g.\cite{An_Fu_Lai_2019,Dutkay-Han-Sun-2009,Dutkay-Lai_2017,An-He_2014,
An-He-Lau_2015,An-He-Tao_2015,Dai_2012,Dai-He-Lai_2013,Dai-He-Lau_2014,
Deng_2014,Deng_2016,Deng-Lau_2015,Dutkay-Haussermann-Lai_2015,
Dutkay-Jorgensen_2007,Dutkay-Jorgensen_2007_1,Fu-Lai_2018,Hu-Lau_2008,
Laba-Wang_2002,Lagarias-Wang_1997,Lai-Wang_2017,Li_2011,Li_2012,
Li_2008,Li_2015,Liu-Dong-Li_2017,Chen-Liu_2019,Liu-Luo_2017,Wang_2002} and the references therein), and some surprising convergence properties of the associated mock Fourier series were discovered \cite{Dutkay_Han_Sun_2014,Strichartz_2006}.

Let $M\in M_n(\mathbb{R})$ be an expanding matrix (that is, all the eigenvalues of $M$ have moduli $>1$), and $D\subset\mathbb{R}^n$ be a finite digit set of cardinality $\left| D\right|$. It is well known that the \emph{self-affine measure} $\mu:=\mu_{M,D}$ is the unique probability measure satisfying
\begin{equation}\label {eq(1.1)}
\mu=\frac{1}{|D|}\sum_{d\in D}\mu\circ\phi_d^{-1},
\end{equation}
and is supported on the nonempty compact set
$$
T(M,D)=\left\{\sum_{k=1}^\infty M^{-k}d: d\in D\right\},$$
where $T(M,D)$ is the \emph{attractor} (or \emph{self-affine} set) of
the affine IFS $\{\phi_d(x):\phi_d(x)=M^{-1}(x+d)\}_{d\in D}$\cite{Hutchinson_1981}.

The following theorem implies that the spectral properties of $\mu_{M,D}$ are invariant under a similarity transformation, and it is usually used to simplify the digit set $D$ or the  expanding matrix $M$.

\begin{thm}\label{th(1.1)} \cite{Dutkay-Jorgensen_2007}
Let $D,\tilde{D}\subset \Bbb R^n$ be two finite digit sets, and let $M,\tilde{M}\in M_n(\Bbb R)$ be two expanding matrices. If there exists a matrix  $B\in M_n(\mathbb{R})$  such that $\tilde{D}=B^{-1}D$ and $\tilde{M}=B^{-1}MB$, then

{\rm(}i{\rm)}\; $E(\Lambda)$ is an orthogonal set of $L^2(\mu_{M,D})$ if and only if $E(B^*\Lambda)$ is an orthogonal set of $L^2(\mu_{\tilde{M},\tilde{D}})$, where $B^*$ denotes the transposed  conjugate of $B$.

{\rm (}ii{\rm)}\; $\mu_{M,D}$ is a spectral measure with spectrum $\Lambda$ if and only if  $\mu_{\tilde{M},\tilde{D}}$ is a spectral measure with spectrum $B^*\Lambda$.
\end{thm}

Theorem \ref{th(1.1)} has been applied extensively by many researchers, but they are hard to simplify the digit set and the  expanding matrix simultaneously. In fact,  in order to simplify the digit set, we usually transform the integer matrix into a non-integer matrix, which brings some difficulties to our discussion. For example, let $M\in M_2(\mathbb{Z})$ and $D=\{(0,0)^t,(1,0)^t,(0,2)^t\}$, we can choose $B=\begin{bmatrix}
1 & 0\\
0& 2 \end{bmatrix}
$ and simplify the digit set $D$ to the well-known digit set $\tilde{D}=B^{-1}D=\{(0,0)^t,(1,0)^t,(0,1)^t\}$. However, the matrix $\tilde{M}=B^{-1}MB$ may not always be an integer matrix. To avoid this flaw, in this paper, we consider the spectrality of self-affine measure under a similarity transformation of general linear group. Before stating
our results, we give some definitions and notations.

To the best of our knowledge, almost all the constructions of singular spectral measures are based on the following Hadamard triple assumption.
\begin{de}
Let $M\in M_n(\mathbb{Z})$ be an expanding matrix, and let $D\subset\mathbb{Z}^n$ be a finite set of cardinality $|D|$. We say that $(M,D)$ is admissible  if there exists a finite set $S\subset \mathbb{Z}^n$ with $|S|=|D|$ such that the  matrix
\begin{equation*}
H=\frac{1}{\sqrt{|S|}}\left[e^{2\pi i <M^{-1}d,s>}\right]_{d\in D,s\in S}
\end{equation*}
is unitary, i.e., $H^*H=I$, where $H^{*}$ denotes the transposed conjugate of $H$.
At this time, we say that  $(M^{-1}D,S)$ is a compatible pair and  $(M,D,S)$ is a Hadamard triple.
\end{de}

It is fairly easy to construct an infinite mutually orthogonal set of exponential functions by using the Hadamard triple assumption. However, it is a much more difficult task to check whether these exponentials form a Fourier basis for $L^2(\mu)$. It has been conjectured that \emph{all Hadamard triples will generate spectral self-affine measures}. The conjecture was solved in dimension one by Laba and Wang \cite{Laba-Wang_2002}. Under various additional conditions, it was also valid in higher dimensions (see \cite{Dutkay-Jorgensen_2007,  Strichartz_1998, Strichartz_2000}). Until recently, Dutkay, Haussermann and Lai \cite{Dutkay-Haussermann-Lai_2015} have completely proved the conjecture.

\begin{thm}\cite{Dutkay-Haussermann-Lai_2015} \label{th(1.3)}
The self-affine measure $\mu_{M,D}$ is a spectral measure if $(M,D)$ is admissible.
\end{thm}

For a prime $p$, let $\mathbb{F}_p:=\mathbb{Z}/p\mathbb{Z}$ denote  the residue class fields.
All nonsingular $n\times n$ matrices over $\mathbb{F}_p$ form a finite group under matrix multiplication,
called the \emph{ general linear group} $GL_n(p)$.

For a finite digit set $D$, let
\begin{equation}\label{eq(1.4.0)}
m_D(x)=\frac 1{\left| D\right|}\sum\limits_{d\in D}{e^{2\pi i\langle d,x\rangle}},\quad x\in\mathbb{R}^n,
\end{equation}
denote the \emph{mask polynomial} of $D$. Define
\begin{equation}\label{eq(1.5.0)}
\mathcal{Z}_{D}^n:=\mathcal{Z}(m_{D})\cap [0, 1)^n,
\end{equation}
where $\mathcal{Z}(m_{D}):=\{x\in\mathbb{R}^n:m_{D}(x)=0\}$.
It is easy to see that $m_D(\cdot)$ is a $\mathbb{Z}^n$-periodic function for $D\subset \mathbb{Z}^n$. In this case,
\begin{equation}\label{eq(1.6.0)}
\mathcal{ Z}(m_D)=\mathcal{Z}_{D}^n+\mathbb{Z}^n.
\end{equation}

For a positive integer $q\geq 2$, let
\begin{equation}\label{eq(1.7)}
E_q^n:=\frac{1}{q}\{(l_1,l_2,\cdots,l_n)^t:0\leq l_1,\cdots,l_n\leq q-1\}  \ \ \ {\rm and} \ \ \ \ \mathring{E}_q^n:=E_q^n\setminus\{0\}.
\end{equation}
There are many well-known digit sets, whose roots of mask polynomial have close connections with $\mathring{E}_q^n$. For example,
$$D_1=\left\{\left( {\begin{array}{*{20}{c}}
0\\
0
\end{array}} \right), \left( {\begin{array}{*{20}{c}}
1\\
0
\end{array}} \right),\left( {\begin{array}{*{20}{c}}
0\\
1
\end{array}} \right)\right\} \quad \text{and} \quad D_2=\left\{\left( {\begin{array}{*{20}{c}}
0\\
0
\end{array}} \right), \left( {\begin{array}{*{20}{c}}
1\\
0
\end{array}} \right),\left( {\begin{array}{*{20}{c}}
0\\
1
\end{array}} \right),\left( {\begin{array}{*{20}{c}}
-1\\
-1
\end{array}} \right)\right\}.
$$
 By a direct calculation, we have
 $$\mathcal{Z}_{D_1}^2=\left\{\left( {\begin{array}{*{20}{c}}
\frac{1}{3}\\
\frac{2}{3}
\end{array}} \right) ,\left( {\begin{array}{*{20}{c}}
\frac{2}{3}\\
\frac{1}{3}
\end{array}} \right)\right\}
 \subset \mathring{E}_3^2 \quad \text{and} \quad \mathcal{Z}_{D_2}^2=\left\{\left( {\begin{array}{*{20}{c}}
\frac{1}{2}\\
0
\end{array}} \right),\left( {\begin{array}{*{20}{c}}
0\\
\frac{1}{2}
\end{array}} \right),\left( {\begin{array}{*{20}{c}}
\frac{1}{2}\\
\frac{1}{2}
\end{array}} \right)\right\}
= \mathring{E}_2^2.
$$

In the following definition, we introduce a similarity transformation of general linear group  $GL_n(p)$, which can be used to study the spectrality for some self-affine measures.

\begin{de}
Let $p$ be a  prime number, $M, \tilde{M} \in M_n(\mathbb{Z})$ be two integer matrices, and  $D, \tilde{D}\subset\mathbb{Z}^n$  be two finite  digit  sets. We say that two pairs $(M,D)$ and $(\tilde{M}, \tilde{D})$ are conjugate in $GL_n(p)$ (through the matrix $A,B\in GL_n(p)$) if there
exist two integer matrices $A,B\in GL_n(p)$ with $AB=I$ $\pmod p$ such that one of the following condition holds:
(i) $\tilde{M}=AMB$ and $D=B\tilde{D}$;
(ii) $\tilde{M}=AMB$ and $\tilde{D}=AD$.
\end{de}

Now we give our first main result.

\begin{thm}\label{th(1.4)}
Suppose that $(M,D)$ and $(\tilde{M}, \tilde{D})$ are conjugate in $GL_n(p)$. If
$\mathcal{Z}_{\tilde{D}}^n\subset \mathring{E}_p^n$ or $\mathcal{Z}_{D}^n\subset \mathring{E}_p^n$, then  $(M,D)$ is admissible if and  only if  $(\tilde{M},\tilde{D})$ is admissible.
\end{thm}

\begin{de}\label{de(1.2)}
Let  $\mu$ be a Borel probability measure with compact support on $\mathbb{R}^n$.  Let $\Lambda$ be a finite or countable subset of $\mathbb{R}^n$, and let $E(\Lambda)$ be defined by \eqref{def:eq_E(A)}. We denote $E(\Lambda)$ by  $E^*(\Lambda)$ if $E(\Lambda)$ is a maximal orthogonal set of exponential functions in $L^2(\mu)$. Let
\begin{align}\label{eq(1.7-1)}
n^*(\mu):=\sup\{ |\Lambda|:  E^*(\Lambda) \hbox { is a maximal orthogonal set} \},
\end{align}
 and call $n^*(\mu)$  the maximal cardinality of the orthogonal exponential functions  in $L^2(\mu)$.
\end{de}

For the conjugated pairs $(M,D)$ and $(\tilde{M}, \tilde{D})$, the following theorem indicates the relationship of orthogonal exponentials  between  $L^2(\mu_{M,D})$ and $L^2(\mu_{\tilde{M},\tilde{D}})$.

\begin{thm}\label{th(1.5)}
Suppose that $(M,D)$ and $(\tilde{M}, \tilde{D})$ are conjugate in $GL_n(p)$. If
$\mathcal{Z}_{\tilde{D}}^n\subset \mathring{E}_p^n$ or $\mathcal{Z}_{D}^n\subset \mathring{E}_p^n$, then $L^2(\mu_{M,D})$ has  infinite many orthogonal exponential functions if and only if $L^2(\mu_{\tilde{M},\tilde{D}})$  has  infinite many orthogonal exponential functions. Moreover, if $\det (M)\notin p\mathbb{Z}$, then $n^*(\mu_{M,D})=n^*(\mu_{\tilde{M},\tilde{D}})\leq p^n$.
\end{thm}

The planar Sierpinski measure  $\mu_{M,\mathcal {D}}$  generated by an expanding integer matrix $M\in M_2(\mathbb{Z})$ and the digit set
\begin{equation}\label{eq(1.9)}
\mathcal {D}=\left\{\left( {\begin{array}{*{20}{c}}
0\\
0
\end{array}} \right), \left( {\begin{array}{*{20}{c}}
1\\
0
\end{array}} \right),\left( {\begin{array}{*{20}{c}}
0\\
1
\end{array}} \right)\right\}
\end{equation}
has attracted many researchers \cite{An-He-Tao_2015,Deng-Lau_2015,Deng_2016,Li_2008,Liu-Dong-Li_2017,Chen-Liu_2019}, and the spectral properties of $\mu_{M,\mathcal {D}}$ have been completely characterized.  For a matrix $M=\left[ {\begin{array}{*{20}{c}}
{{a}}&{{b}}\\
{{d}}&{{{c}}}
\end{array}} \right]\in M_2(\mathbb{Z})$, by using the residue system of modulo $3$, it can be rewritten in the following form:
\begin{equation*}
M=\left[ {\begin{array}{*{20}{c}}
{a}&{b}\\
{d}&{c}
\end{array}} \right]=3\tilde{M}+M_\alpha,
\end{equation*}
where $\tilde{M}\in M_2(\mathbb{Z})$ and the entries of the matrix $M_\alpha$ are $0, 1$ or $-1$.

Before stating the results about the Sierpinski measure $\mu_{M,\mathcal {D}}$, we first define the following two sets:
\begin{equation}\label{eq(1.11)}
\mathfrak{M}_1:= \pm\left\{\left[ {\begin{array}{*{20}{c}}
0&0\\
0&0
\end{array}} \right], \left[ {\begin{array}{*{20}{c}}
1&1\\
1&1
\end{array}} \right], \left[ {\begin{array}{*{20}{c}}
0&1\\
0&1
\end{array}} \right], \left[ {\begin{array}{*{20}{c}}
1&0\\
1&0
\end{array}} \right],  \left[ {\begin{array}{*{20}{c}}
1&-1\\
1&-1
\end{array}} \right]     \right\}
\end{equation}
and
\begin{equation}\label{eq(1.12)}
\mathfrak{M}_2=:\pm\left\{ \left[ {\begin{array}{*{20}{c}}
0&1\\
1&1
\end{array}} \right], \left[ {\begin{array}{*{20}{c}}
0&1\\
1&-1
\end{array}} \right], \left[ {\begin{array}{*{20}{c}}
1&1\\
1&0
\end{array}} \right],  \left[ {\begin{array}{*{20}{c}}
1&-1\\
1&1
\end{array}} \right],\left[ {\begin{array}{*{20}{c}}
1&-1\\
-1&0
\end{array}} \right],\left[ {\begin{array}{*{20}{c}}
1&1\\
-1&1
\end{array}} \right]       \right\}.
\end{equation}
Now we summarize the known results of the Sierpinski measure $\mu_{M,\mathcal {D}}$ as follows:
\begin{thm} \cite{An-He-Tao_2015, Liu-Dong-Li_2017}\label{th(1.6)}
Let $M\in M_2(\mathbb{Z})$ be an expanding matrix, and let $\mathcal {D}$  and $\mu_{M,\mathcal {D}}$ be defined by \eqref{eq(1.9)} and \eqref{eq(1.1)}, respectively. Then

(i) \ \  $\mu_{M,\mathcal {D}}$ is a spectral measure $\Longleftrightarrow$ $(M, \mathcal {D})$ is admissible  $\Longleftrightarrow$ $M\in\mathfrak{M}_1 \pmod 3$ ;

(ii) \ \ if $\det (M)\notin 3\mathbb{Z}$, then $L^2(\mu_{M,\mathcal {D}})$ contains at most $9$ mutually
orthogonal exponential functions, and the upper boundary $9$ can be obtained  if and only if  $M\in\mathfrak{M}_2  \pmod 3 $.
\end{thm}

By Theorem \ref{th(1.5)}, we can easily extend the conclusion of Theorem \ref{th(1.6)}(ii) to the following general  form.

\begin{cor}\label{th(1.7)}
Let $M\in M_2(\mathbb{Z})$ be an  expanding matrix and $D=\{(0,0)^t,(\alpha_1,\alpha_2)^t,(\beta_1,\beta_2)^t\}$  be an integer digit set  with $\alpha_1\beta_2-\alpha_2\beta_1\notin 3\mathbb{Z}$. If $\det(M)\notin 3\mathbb{Z}$, then $L^2(\mu_{M,D})$ contains at most $9$ mutually
orthogonal exponential functions, and the number $9$ is the best. Moreover, the best number can be attained if and only if $AMB \in \mathfrak{M}_2 \pmod 3$, where $B=\left[ {\begin{array}{*{20}{c}}
\alpha_1&\beta_1\\
\alpha_2&\beta_2
\end{array}} \right]$ and $A\in GL_2(3)$ satisfy $AB=I \pmod 3$.
\end{cor}

For the equivalences of Theorem \ref{th(1.6)}(i), the similar conclusions can also be obtained for the digit set $D=\{(0,0)^t,(\alpha_1,\alpha_2)^t,(\beta_1,\beta_2)^t\}$ with $\alpha_1\beta_2-\alpha_2\beta_1\notin 3\mathbb{Z}$.

\begin{thm}\label{th(1.8)}
Let $M\in M_2(\mathbb{Z})$ be an expanding matrix and $D=\{(0,0)^t,(\alpha_1,\alpha_2)^t,(\beta_1,\beta_2)^t\}$  be an integer digit set  with $\alpha_1\beta_2-\alpha_2\beta_1\notin 3\mathbb{Z}$, and let $\mu_{M,D}$ be defined by \eqref{eq(1.1)}. Then the following statements  are equivalent:

(i) \ \  $\mu_{M,D}$ is a spectral measure;

(ii) \ \ $(M, D)$ is admissible;

(iii) \ \  $(AMB)^*(1,-1)^t\in 3\mathbb{Z}^2$,
where $B=\left[ {\begin{array}{*{20}{c}}
\alpha_1&\beta_1\\
\alpha_2&\beta_2
\end{array}} \right]$ and $A\in GL_2(3)$ satisfy  $AB=I\pmod 3$.
\end{thm}

The organization of the paper is as follows. We give several preparatory definitions and  conclusions in Section 2, and prove Theorem \ref{th(1.5)}, Theorem \ref{th(1.6)} and Corollary  \ref{th(1.7)} in Section 3. As a main application, we complete the proof of  Theorem \ref{th(1.8)} in Section 4.

\section{\bf Preliminaries\label{sect.2}}

In this section, we give some  preliminary definitions and
propositions. For the self-affine measure $\mu_{M,D}$ given in \eqref{eq(1.1)}, the Fourier transform of $\mu_{M,D}$ is defined as usual,
\begin{equation}\label{1.2}
\hat{\mu}_{M,D}(\xi)=\int e^{2\pi i\langle x,\xi\rangle }d\mu_{M,D}(x)=\prod_{j=1}^\infty m_D({M^{*}}^{-{j}}\xi), \quad  \xi\in\mathbb{R}^n,
\end{equation}
where $M^*$ denotes the transposed  conjugate of $M$, and $m_D(x)$ is the \emph{mask polynomial} of $D$ defined by \eqref{eq(1.4.0)}.

Let $\mathcal{Z}(f)$ denote the zero set of the function $f$, i.e., $\mathcal{Z}(f)=\{x\in\mathbb{R}^n:f(x)=0\}$.
It follows from \eqref{1.2} that
\begin{equation}\label{1.3}
\mathcal{Z}(\hat{\mu}_{M, D})=\bigcup_{j=1}^{\infty}M^{*j}\mathcal{Z}(m_D).
\end{equation}
For any $\lambda_1, \lambda_2\in \mathbb{R}^n$, $\lambda_1\neq \lambda_2$, the orthogonality condition
$$
0=\langle e^{2\pi i \langle \lambda_1,x\rangle},e^{2\pi i \langle \lambda_2,x\rangle}\rangle_{L^2(\mu_{M,D})}=\int e^{2\pi i \langle\lambda_1-\lambda_2,x\rangle}d\mu_{M,D}(x)=\hat{\mu}_{M,D}(\lambda_1-\lambda_2)
$$
relates to $\mathcal{Z}(\hat{\mu}_{M,D})$ directly. It is easy to see that for a countable subset $\Lambda\subset\mathbb{R}^n$,  $E(\Lambda)=\{e^{2\pi i\langle\lambda,x\rangle}:\lambda\in\Lambda\}$ is an orthogonal family of $L^2(\mu_{M,D})$ if and only if
 \begin{equation}\label{1.4}
 (\Lambda-\Lambda)\setminus\{0\}\subset\mathcal{Z}(\hat{\mu}_{M,D}).
\end{equation}

For a positive number $m$, let $\varphi(m)$ denote the \emph{Euler's phi function} which  equal to the number of integers in the set $\{1,2,\ldots, m-1\}$ that are relatively prime to $m$.  For more information about the Euler's phi function, the readers can refer to \cite{Nathanson_1996}. The following lemma is the famous \emph{Euler's theorem}.

\begin{lemma}\cite{Nathanson_1996}\label{th(2.1)}
Let $m$ be a positive integer, and let $N$ be an integer relatively prime to $m$. Then $N^{\varphi(m)}=1\pmod m.$
\end{lemma}

\begin{de}\label{de(2.9)}
Let $M\in GL_n(p)$, then the least positive integer $e$ for which $M^e=I(\mod p)$ is called the order of
$M$ and denoted by $O_p(M)$.
\end{de}

\begin{lemma} \cite{Ghorpade-Hasan-Kumari_2011}\label{th(2.4)}
Let $M\in GL_n(p)$, then $O_p(M)\leq p^n-1$.
\end{lemma}

In order to prove Theorem \ref{th(1.5)}, we need the following Lemmas. In \cite{Li_2015_1}, Li gave a sufficient and necessary condition for the infinite $\mu_{M,D}$-orthogonality.

\begin{lemma}\cite{Li_2015_1}\label{th(2.5.1)}
For an expanding integer matrix $M\in M_n(\mathbb{Z})$ and a finite digit set $D\subset \mathbb{Z}^n$, let $\mu_{M,D}$ and $\mathcal{Z}_{D}^n$ be defined by \eqref{eq(1.1)} and \eqref{eq(1.5.0)}, respectively. Suppose that  $\mathcal{Z}_{D}^n\subset \mathbb{Q}^n $ is a finite set, then $L^2(\mu_{M,D})$ contains an infinite orthogonal set of exponential functions if and only if $M^{*j}(\mathcal{Z}_{D}^n)\cap \mathbb{Z}^n\neq\emptyset$ for some $j\in \mathbb{N}$.
\end{lemma}
\begin{lemma} \label{lem(2.5)}
Let $M\in M_n(\mathbb{Z})$ be an expanding matrix with $\det(M)=L$, and let $\mathring{E}_p^n$ be defined by \eqref{eq(1.7)}.  Then

(i) If $\gcd(L,p)=1$, then $M\mathring{E}_p^n=\mathring{E}_p^n  \pmod {\Bbb Z^n}$ and $L^{\varphi(p)j}(M^{j}\lambda+\Bbb Z^n)\subset M^{j}(\lambda+\Bbb Z^n)$ for any integer $j\geq 0$ and  $\lambda \in \mathring{E}_{p}^n$, where $\varphi$ is the Euler's phi function.

(ii) $A(\lambda+\Bbb Z^n)\subset\lambda+\Bbb Z^n$ for any $\lambda \in \mathring{E}_{p}^n$ and any integer matrix $A$ with $A=I \pmod p$.
\end{lemma}
\begin{proof} (i) In order to get $M\mathring{E}_p^n=\mathring{E}_p^n  \pmod {\mathbb{Z}^n}$, we first show that $M\mathring{E}_p^n\cap\mathbb{Z}^n=\emptyset$. Suppose there exist
$\xi\in \mathring{E}_p^n$  and  an integer vector $v\in \mathbb{Z}^n$ such that $M\xi=v$, then
$$
L\xi=LM^{-1}M\xi=LM^{-1}v.
$$
Since $LM^{-1}$ is  an  integer matrix, the right side of the above equation is an integer vector. However, the left side cannot be an integer vector since $\gcd(L,p)=1$. This contradiction shows that $M\mathring{E}_p^n\cap\mathbb{Z}^n=\emptyset$, which implies that $M\mathring{E}_p^n\subset\mathring{E}_p^n  \pmod {\mathbb{Z}^n}$. If $M\mathring{E}_p^n\subsetneqq\mathring{E}_p^n  \pmod {\mathbb{Z}^n}$, there must exist different $\lambda_1,\lambda_2\in\mathring{E}_p^n$ such that
$M\lambda_1=M\lambda_2 \in \mathring{E}_p^n \pmod {\mathbb{Z}^n}$. This implies that $M(\lambda_1-\lambda_2)\in \mathbb{Z}^n$. However, it contracts the fact $M\mathring{E}_p^n\cap\mathbb{Z}^n=\emptyset$ since $\lambda_1-\lambda_2\in \mathring{E}_p^n  \pmod {\mathbb{Z}^n}$. Therefore $M\mathring{E}_p^n=\mathring{E}_p^n  \pmod {\mathbb{Z}^n} $.

It follows from Lemma \ref{th(2.1)} and $\gcd(L,p)=1$ that there exist two integers $n$ and $m$ such that  $L^{\varphi(p)j}=(np+1)^j=pm+1$ for $j\geq0$. Hence for any $\lambda \in \mathring{E}_{p}^n$,  by using $L^{\varphi(p)j}M^{^{-j}}\in M_n(\mathbb{Z})$ and $pm\lambda\in\mathbb{Z}^n$, we have
\begin{align*}
L^{\varphi(p)j}(M^{j}\lambda+\Bbb Z^n)&= L^{\varphi(p)j}(M^{j}\lambda+M^{j}M^{^{-j}}\Bbb Z^n)\\
&=M^{j}((pm+1)\lambda+L^{\varphi(p)j}M^{^{-j}}\Bbb Z^n)\\
&=M^{j}(\lambda+pm\lambda+L^{\varphi(p)j}M^{^{-j}}\Bbb Z^n)\\
&\subset M^{j}(\lambda+\Bbb Z^n).
\end{align*}

(ii) Since $A=I  \pmod p$, there exists $M_0\in M_n(\mathbb{Z})$ such that $A=pM_0+I$.  Then for any $\lambda \in \mathring{E}_{p}^n$, we have
\begin{align*}
A(\lambda+\Bbb Z^n)&=(pM_0+I)(\lambda+\Bbb Z^n)\\
&=\lambda+pM_0\lambda+(pM_0+I)\Bbb Z^n\\
&\subset\lambda+\Bbb Z^n.
\end{align*}
The proof of Lemma \ref{lem(2.5)}  is completed.
\end{proof}

The following proposition  reflects the  relationships  between the two sets ${M^*}^{j}\mathcal{Z}(m_D)$ and ${\tilde{M^*}}^{j}\mathcal {Z}(m_{\tilde{D}})$ under the matrix transformation, when $(M,D)$ and $(\tilde{M}, \tilde{D})$ are conjugate in $GL_n(p)$.
\begin{pro}\label{lem(4.3)}
Suppose that $(M,D)$ and $(\tilde{M}, \tilde{D})$ are conjugate in $GL_n(p)$ (through the matrix $A,B\in GL_n(p)$). Then the following  statements  hold:

(i) If $\mathcal{Z}_{\tilde{D}}^n\subset \mathring{E}_p^n$, then for any $j\geq1$,
$${M^*}^{j}\mathcal{Z}(m_D)\subset
\frac{1}{(\det(AB))^{j+1}}
A^*({\tilde{M^*}}^{j}\mathcal {Z}(m_{\tilde{D}})+\Bbb Z^n).$$
In particular, for $j=1$, we also have
$${M^*}\mathcal {Z}(m_D)\subset\frac{1}{\det (AB)}{B^*}^{-1}\tilde{M}^*\mathcal {Z}(m_{\tilde{D}})\quad \mbox{and} \quad
{{\tilde{M}^*}}\mathcal{Z}(m_{\tilde{D}})\subset B^*{M^*}\mathcal{Z}(m_D).$$

(ii) If $\det(M)\notin p\mathbb{Z}$ and $\mathcal{Z}_{\tilde{D}}^n\subset \mathring{E}_p^n$, then
$$c_1B^* \mathcal{Z}(\hat{\mu}_{M, D})\subset \mathcal{Z}(\hat{\mu}_{\tilde{M}, \tilde{D}})
 \quad \text{and} \quad  c_2A^* \mathcal{Z}(\hat{\mu}_{\tilde{M}, \tilde{D}})\subset \mathcal{Z}(\hat{\mu}_{M, D}),$$
 where $c_1=\det(AB){|\det(\tilde{M}) |}^{(p-1)(p^n-1)}$ and $c_2={|\det({M})  |}^{(p-1)(p^n-1)}$.
\end{pro}

\begin{proof} Without loss of generality, we can let $\tilde{M}=AMB$ and $D=B\tilde{D}$ (the other  situation is similar).
	
(i) Note first that
\begin{align*}
\det(AB){A^*}^{-1}{B^*}^{-1}= \mbox{adjoint matrix of\;} (AB)^*=
I\pmod p,
\end{align*}
where the last equality follows from $AB=I\pmod p$, hence there exists  some $M'\in M_n(\mathbb{Z})$ such that $\det(AB){A^*}^{-1}{B^*}^{-1}=pM'+I$.
Therefore, for any $j\geq 1$,
\begin{align}\label{eq(2.5.0)} \nonumber
{M^*}^{j}\mathcal {Z}(m_D)&={M^*}^{j}A^*{A^*}^{-1}{B^*}^{-1}B^* \mathcal {Z}(m_D)\\
&=\frac{1}{\det (AB)}{M^*}^{j}A^*(pM'+I)B^* \mathcal {Z}(m_D).
\end{align}
From $D=B\tilde{D}$, it is easy to see that
$$
m_{\tilde{D}}(x)=\frac 1{\left| \tilde{D}\right|}\sum_{\tilde{d}\in \tilde{D} } e^{2\pi i\langle\tilde{d},x\rangle}
=\frac 1{\left| D\right|}\sum_{d\in D } e^{2\pi i\langle d,{B^*}^{-1}x\rangle}=m_D({B^*}^{-1}x) \;\;\mbox{for all} \;\;x\in\Bbb R^n.
$$
This shows that
\begin{align}\label{eq(3.11.1)}
B^* \mathcal {Z}(m_D)=\mathcal {Z}(m_{\tilde{D}})\subset \mathring{E}_p^n+\Bbb Z^n.
\end{align}
Thus by \eqref{eq(2.5.0)}, \eqref{eq(3.11.1)} and Lemma \ref{lem(2.5)} (ii), an iterative calculation gives that
\begin{align}\label{lem4.3-eq2}
{M^*}^{j}\mathcal {Z}(m_D)&=\frac{1}{\det (AB)}{M^*}^{j}A^*(pM'+I)B^* \mathcal {Z}(m_D)\nonumber\\
&\subset \frac{1}{\det (AB)}{M^*}^{j-1}M^*A^*\mathcal {Z}(m_{\tilde{D}})\nonumber\\ \nonumber
&=\frac{1}{\det (AB)}{M^*}^{j-1}A^*{A^*}^{-1}{B^*}^{-1}B^*M^*A^*\mathcal {Z}(m_{\tilde{D}})\\ \nonumber
&=\frac{1}{(\det(AB))^{2}}{M^*}^{j-1}A^*(pM'+I)\tilde{M}^*\mathcal {Z}(m_{\tilde{D}})\\
&\subset \frac{1}{(\det(AB))^{2}}{M^*}^{j-1}A^*\left(\tilde{M}^*\mathcal {Z}(m_{\tilde{D}})+\Bbb Z^n\right) \nonumber\\
&\subset \cdots\nonumber\\
&\subset \frac{1}{(\det(AB))^{j+1}}A^*\left({\tilde{M^*}}^{j}\mathcal {Z}(m_{\tilde{D}})+\Bbb Z^n\right).
\end{align}

For $j=1$, by \eqref{eq(2.5.0)} and Lemma \ref{lem(2.5)} (ii), a direct calculation also gives that
\begin{align*}
{M^*}\mathcal {Z}(m_D)&=\frac{1}{\det (AB)}{M^*}A^*(pM'+I)B^* \mathcal {Z}(m_D)\nonumber\\
&\subset \frac{1}{\det (AB)}M^*A^*\mathcal {Z}(m_{\tilde{D}})\nonumber\\ \nonumber
&=\frac{1}{\det (AB)}{B^*}^{-1}B^*M^*A^*\mathcal {Z}(m_{\tilde{D}})\\ \nonumber
&=\frac{1}{\det (AB)}{B^*}^{-1}\tilde{M}^*\mathcal {Z}(m_{\tilde{D}}).
\end{align*}
On the other hand, note that $B^*A^*\mathcal{Z}(m_{\tilde{D}})\subset \mathcal{Z}(m_{\tilde{D}})=B^*\mathcal{Z}(m_D)$, so that
\begin{align}\label{eq:2.10.0}
\mathcal{Z}(m_{\tilde{D}})\subset {A^*}^{-1}\mathcal{Z}(m_D).
\end{align}
Hence, ${{\tilde{M}^*}}\mathcal{Z}(m_{\tilde{D}})\subset  B^*M^*A^*{A^*}^{-1}\mathcal{Z}(m_{D})= B^* M^*\mathcal{Z}(m_D).$

(ii)
It follows from $AB=I\pmod p$ that $A^*B^*=pM_0+I$ for some $M_0\in M_n(\mathbb{Z})$.
Hence for any $j\geq 1$,  there exists a matrix $M_{j}\in M_n(\mathbb{Z})$ such that
\begin{align}\label{eq(2.10)} \nonumber
A^*\tilde{M^*}^jB^*=A^*{(B^*M^*A^*)}^{j}B^*&=A^*B^*M^*A^*B^*M^*\cdots A^*B^*M^*A^*B^*\\ \nonumber
&=(pM_0+I)M^*(pM_0+I)M^*\cdots(pM_0+I)M^*(pM_0+I)\\
&={M^*}^j+pM_{j}.
\end{align}
Thus
\begin{align*}
\det(AB)B^*{M^*}^j\mathcal{Z}(m_{D})&=
\det(AB)B^*(A^*\tilde{M^*}^jB^*-pM_{j})\mathcal{Z}(m_{D})\\
&\subset\det(AB)B^*A^*\tilde{M^*}^jB^*\mathcal{Z}(m_{D})-\det(AB) pB^*M_{j}{B^*}^{-1}B^*\mathcal{Z}(m_{D})\\
&\subset \det(AB)B^*A^*\tilde{M^*}^j \mathcal{Z}(m_{\tilde{D}})+\Bbb Z^n.
\end{align*}
where the last inclusion relation follows from \eqref{eq(3.11.1)} and the fact $\det(AB){B^*}^{-1}\in M_n(\mathbb{Z})$. Then using $AB=I \pmod p$ again,  we have
\begin{align}\label{eq(2.11)}
\det(AB)B^*{M^*}^j\mathcal{Z}(m_{D})\subset \tilde{M^*}^j \mathcal{Z}(m_{\tilde{D}})+\Bbb Z^n.
\end{align}

Let $N:=|\det(M)|$ and $\tilde{N}:=|\det (\tilde{M})|$. Note that $\tilde{M}=AMB$ for some $A,B\in GL_n(p)$. Then $N, \tilde{N}\notin p\mathbb{Z}$ and $\tilde{M^*}\in GL_n(p)$. By using Lemma \ref{th(2.4)}, we have $\tau:=O_p(\tilde{M^*})\leq p^n-1$, where $O_p(\tilde{M^*})$ is the order of
$\tilde{M^*}$. It is clear that for any $j\geq 1$, there exist some integers $m$ and $1\leq j'\leq \tau$ such that $j=m \tau+j'$, hence the definition of the order of matrix implies that
\begin{align}\label{eq(2.12)}
\tilde{M^*}^j\mathcal{Z}(m_{\tilde{D}})+\mathbb{Z}^n= \tilde{M^*}^{j'}\mathcal{Z}(m_{\tilde{D}})+\mathbb{Z}^n.
\end{align}
In addition, it follows from $\gcd(\tilde{N},p)=1$ and Lemma \ref{th(2.1)} that $\tilde{N}^{\varphi(p)}=1\pmod p$, then there  exists an integer $\iota$ such that
$\tilde{N}^{\varphi(p)(p^n-1-j')}=\iota p+1$.  Now using this and \eqref{eq(2.12)}, multiplying both sides of \eqref{eq(2.11)} by $\tilde{N}^{\varphi(p)(p^n-1)}$, we obtain
\begin{align*}
\tilde{N}^{\varphi(p)(p^n-1)}\det(AB)B^*{M^*}^j\mathcal{Z}(m_{D})
&\subset \tilde{N}^{\varphi(p)(p^n-1)}(\tilde{M^*}^j \mathcal{Z}(m_{\tilde{D}})+\Bbb Z^n)\\ \nonumber
&= \tilde{N}^{\varphi(p)j'}(\iota p+1)(\tilde{M^*}^{j'} \mathcal{Z}(m_{\tilde{D}})+\Bbb Z^n)\\ \nonumber
&\subset \tilde{N}^{\varphi(p)j'}(\tilde{M^*}^{j'} \mathcal{Z}(m_{\tilde{D}})+\Bbb Z^n).
\end{align*}
Note that $c_1=\tilde{N}^{(p-1)(p^n-1)}\det(AB)$. Thus by Lemma \ref{lem(2.5)}(i) and the fact $\varphi(p)=p-1$ for prime $p$, we have
\begin{align*}
c_1B^*{M^*}^j\mathcal{Z}(m_{D}) \subset\tilde{M^*}^{j'}(\mathcal{Z}(m_{\tilde{D}})+\Bbb Z^n)=\tilde{M^*}^{j'}\mathcal{Z}(m_{\tilde{D}}).
\end{align*}
Therefore, \eqref{1.3} implies that
\begin{align}\label{eq(2.13.1)}
c_1B^* \mathcal{Z}(\hat{\mu}_{M, D})&=\bigcup_{j=1}^\infty c_1B^*{M^*}^j\mathcal{Z}(m_{D}) \subset \bigcup_{j'=1}^\tau\tilde{M^*}^{j'}\mathcal{Z}(m_{\tilde{D}})\subset \mathcal{Z}(\hat{\mu}_{\tilde{M}, \tilde{D}}).
\end{align}

Now we will prove $c_2A^* \mathcal{Z}(\hat{\mu}_{\tilde{M}, \tilde{D}})\subset \mathcal{Z}(\hat{\mu}_{M, D})$, where $c_2=N^{(p-1)(p^n-1)}$.  Note first that $B^*A^*=I \pmod p$ and $\mathcal{Z}(m_{\tilde{D}})\subset \mathring{E}_p^n+\Bbb Z^n$. Applying Lemma \ref{lem(2.5)}(ii), we have
\begin{align}\label{eq(2.15.0)}
\mathcal{Z}(m_{\tilde{D}})=B^*A^*\mathcal{Z}(m_{\tilde{D}})+\Bbb Z^n.
\end{align}
Then it follows from \eqref{eq(2.15.0)} and \eqref{eq(2.10)} that, for any $j\geq 1$,
\begin{align}\label{eq:2.15}
A^*\tilde{M^*}^{j}\mathcal{Z}(m_{\tilde{D}})&=A^*\tilde{M^*}^{j}\left(B^*A^*\mathcal{Z}(m_{\tilde{D}})+\Bbb Z^n\right)\nonumber\\
&=({M^*}^j+pM_{j})A^*\mathcal{Z}(m_{\tilde{D}})+A^*\tilde{M^*}^{j}\Bbb Z^n\nonumber\\
&\subset {M^*}^jA^*\mathcal{Z}(m_{\tilde{D}})+\Bbb Z^n.
\end{align}
And again, just like in the previous case, there exists an integer $\kappa$  such that $N^{\varphi(p)(p^n-1-j'')}=p\kappa+1$, where $j''=j\pmod {O_p(M^*)}$ satisfies $1\leq j''\leq O_p(M^*)$. Multiplying both sides of \eqref{eq:2.15} by $c_2=N^{\varphi(p)(p^n-1)}$, and using the fact $A^*\mathcal{Z}(m_{\tilde{D}})\subset\mathring{E}_p^n+\Bbb Z^n$ and Lemma \ref{lem(2.5)}(i), we obtain that
\begin{align*}
c_2A^*\tilde{M^*}^{j}\mathcal{Z}(m_{\tilde{D}})&\subset N^{\varphi(p)j''}(p\kappa+1)\left({M^*}^jA^*\mathcal{Z}(m_{\tilde{D}})+\Bbb Z^n\right)\\
&\subset N^{\varphi(p)j''}\left({M^*}^jA^*\mathcal{Z}(m_{\tilde{D}})+\Bbb Z^n\right)\\
&=N^{\varphi(p)j''}\left({M^*}^{j''}A^*\mathcal{Z}(m_{\tilde{D}})+\Bbb Z^n\right)\\
&\subset {M^*}^{j''}\left(A^*\mathcal{Z}(m_{\tilde{D}})+\Bbb Z^n\right)\\
&\subset{M^*}^{j''}\mathcal{Z}(m_{D}),
\end{align*}
where the last inclusion relation follows from  \eqref{eq:2.10.0}. Similar to \eqref{eq(2.13.1)},
we have $c_2A^* \mathcal{Z}(\hat{\mu}_{\tilde{M}, \tilde{D}})\subset \mathcal{Z}(\hat{\mu}_{M, D})$.

This completes the proof of Proposition \ref{lem(4.3)}.
\end{proof}

Similar to Proposition \ref{lem(4.3)}, by symmetry, we have
\begin{re}\label{lem(4.4)}
Suppose that $(M,D)$ and $(\tilde{M}, \tilde{D})$ are conjugate in $GL_n(p)$ (through the matrix $A,B\in GL_n(p)$). Then the following  statements  hold:

(i) If $\mathcal{Z}_{D}^n\subset \mathring{E}_p^n$, then for any $j\geq1$,
$${{\tilde{M}}^j}\mathcal{Z}(m_{\tilde{D}})
\subset \frac{1}{(\det(AB))^{j+1}}B^*({M^*}^j\mathcal{Z}(m_D)+\Bbb Z^n).$$
In particular, for $j=1$, we also have
$${M^*}\mathcal{Z}(m_D)\subset A^*{{\tilde{M}}^{*}}\mathcal{Z}(m_{\tilde{D}})
\quad\mbox{and}\quad{{\tilde{M}}^{*}}\mathcal{Z}(m_{\tilde{D}})\subset \frac{1}{\det (AB)}{A^*}^{-1}{M^*}\mathcal{Z}(m_D).$$

(ii) If $\det (M)\notin p\mathbb{Z}$ and $\mathcal{Z}_{D}^n\subset \mathring{E}_p^n$, then
$$d_1B^*\mathcal{Z}(\hat{\mu}_{M, D})\subset \mathcal{Z}(\hat{\mu}_{\tilde{M}, \tilde{D}}) \quad \text{and} \quad d_2A^*\mathcal{Z}(\hat{\mu}_{\tilde{M}, \tilde{D}})\subset \mathcal{Z}(\hat{\mu}_{M, D}),$$
where $d_1={|\det (\tilde{M}) |}^{(p-1)(p^n-1)}$ and $d_2=\det(AB){|\det (M) |}^{(p-1)(p^n-1)}$.
\end{re}

\section{\bf The spectrality under the similarity transformation  of $GL_n(p)$ \label{sect.3}}

In this section, we mainly consider the spectrality of self-affine measure after a similarity transformation  of $GL_n(p)$.
We will show that, for some matrices $M$ and digit sets $D$, the admissible property  of  $(M,D)$ and the cardinality of orthogonal exponentials of $L^2(\mu_{M,D})$ are invariant under a  similarity transformation  of $GL_n(p)$.

\begin{thm}\label{th(3.1)}
Suppose that $(M,D)$ and $(\tilde{M}, \tilde{D})$ are conjugate in $GL_n(p)$. If
$\mathcal{Z}_{\tilde{D}}^n\subset \mathring{E}_p^n$ or $\mathcal{Z}_{D}^n\subset \mathring{E}_p^n$, then  $(M,D)$ is admissible if  and  only if  $(\tilde{M},\tilde{D})$ is admissible.
\end{thm}

\begin{proof} Without loss of generality, we let   $\tilde{M}:=AMB$ and $D=B\tilde{D}$ with $\mathcal{Z}_{\tilde{D}}^n\subset \mathring{E}_p^n$ (Other situations are similar).

Necessity:  \ \  Suppose $(M,D)$ is admissible, then there exists $S\subset \mathbb{Z}^n$ with $|S|=|D|$  such that $(M^{-1}D,S)$ is  a compatible pair, i.e., the matrix
\begin{align} \nonumber
H=\frac{1}{\sqrt{|S|}}\left[e^{2\pi i <M^{-1}d,s>}\right]_{d\in D,s\in S}
=\frac{1}{\sqrt{|S|}}\left[e^{2\pi i <d,{M^*}^{-1}s>}\right]_{d\in D,s\in S} \quad\mbox{is unitary.}
\end{align}
This is equivalent to $({M^*}^{-1}S-{M^*}^{-1}S)\setminus\{0\}\subset \mathcal {Z}(m_{D})$, i.e., $(S-S)\setminus\{0\}\subset M^*\mathcal {Z}(m_{D}).$
It follows from Proposition \ref{lem(4.3)}(i) that
$(S-S)\setminus\{0\}\subset M^*\mathcal {Z}(m_{D}) \subset\frac{1}{\det(AB)}{B^*}^{-1}\tilde{M}^*\mathcal{Z}(m_{\tilde{D}}).$
This implies that
$$\det(AB){B^*}(S-S)\setminus\{0\} \subset\tilde{M}^*\mathcal{Z}(m_{\tilde{D}}),$$
Hence, it is easy to see  that  $\tilde{S}:=\det(AB){B^*}S\subset \Bbb Z^n$ and $(\tilde{M}^{-1}\tilde{D},\tilde{S})$ is a compatible pair, i.e.,  $(\tilde{M},\tilde{D})$ is admissible.

Sufficiency:  \ \  If $(\tilde{M},\tilde{D})$ is admissible, then there exists  $\tilde{S} \subset \mathbb{Z}^n$  with  $|\tilde{S}|=|\tilde{D}|$ such that $(\tilde{M}^{-1}\tilde{D},\tilde{S})$ is a compatible pair.  Similar to the previous process, we have
\begin{align}\label{eq(3.7.2)}
(\tilde{S}-\tilde{S})\setminus\{0\}\subset \tilde{M^*}\mathcal {Z}(m_{\tilde{D}}).
\end{align}
Since $\gcd(\det(B),p)=1$, Lemma \ref{th(2.1)} gives
$|\det(B)|^{\varphi(p)}=1  \pmod p$,
where $\varphi$ is Euler's phi function.
Hence, it follows from Lemma \ref{lem(2.5)}(ii) that
\begin{align*}
|\det(B)|^{\varphi(p)}(\tilde{S}-\tilde{S})\setminus\{0\}\subset |\det(B)|^{\varphi(p)}\tilde{M^*}\mathcal {Z}(m_{\tilde{D}})\subset \tilde{M^*}\mathcal {Z}(m_{\tilde{D}}),
\end{align*}
so Proposition \ref{lem(4.3)}(i) implies
$|\det(B)|^{\varphi(p)}(\tilde{S}-\tilde{S})\setminus\{0\}\subset B^*{M^*}\mathcal{Z}(m_D)$, i.e.,
$$
|\det(B)|^{\varphi(p)}{B^*}^{-1}(\tilde{S}-\tilde{S})\setminus\{0\}\subset {M^*}\mathcal{Z}(m_D).
$$
Obviously,  $S:=|\det(B)|^{\varphi(p)}{B^*}^{-1}\tilde{S}\subset\Bbb Z^n$  and $(M^{-1}D,S)$ is  a compatible pair  since $|\det(B)|^{\varphi(p)}{B^*}^{-1}$ is an integer matrix. Therefore, $(M,D)$ is admissible.

Hence we complete the proof of Theorem \ref{th(3.1)}.
\end{proof}

\begin{thm}\label{th(3.2)}
Suppose that $(M,D)$ and $(\tilde{M}, \tilde{D})$ are conjugate in $GL_n(p)$. If
$\mathcal{Z}_{\tilde{D}}^n\subset \mathring{E}_p^n$ or $\mathcal{Z}_{D}^n\subset \mathring{E}_p^n$, then $L^2(\mu_{M,D})$  has  infinite many orthogonal exponential functions if and only if $L^2(\mu_{\tilde{M},\tilde{D}})$  has  infinite many orthogonal exponential functions. Moreover, if $\det (M)\notin p\mathbb{Z}$, then $n^*(\mu_{M,D})=n^*(\mu_{\tilde{M},\tilde{D}})\leq p^n$.
\end{thm}

\begin{proof} Without loss of generality, we let   $\tilde{M}:=AMB$ and $D=B\tilde{D}$ with $\mathcal{Z}_{\tilde{D}}^n\subset \mathring{E}_p^n$ (Other situations are similar).
It follows from  \eqref{eq(1.6.0)}  and \eqref{eq(3.11.1)} that
$$
\mathcal{Z}_{D}^n+\Bbb Z^n=\mathcal{Z}(m_D)=B^{*-1}\mathcal{Z}(m_{\tilde{D}})=B^{*-1}(\mathcal{Z}_{\tilde{D}}^n+\Bbb Z^n)\subset B^{*-1}(\mathring{E}_p^n+\Bbb Z^n)\subset \mathring{E}_{p|\det(B)|}^n+\Bbb Z^n.
$$
So $\mathcal{Z}_{D}^n\subset\mathring{E}_{p|\det(B)|}^n\subset \Bbb Q^n$, and it is a finite set. Then by Lemma  \ref{th(2.5.1)}, in order to conclude that $L^2(\mu_{M,D})$  has  infinite many orthogonal exponential functions if and only if $L^2(\mu_{\tilde{M},\tilde{D}})$  has  infinite many orthogonal exponential functions, it is sufficient to prove that
$$\exists\;\; j\in \Bbb N\;\; s.t. \;\;(M^{*j}\mathcal{Z}_{D}^n)\cap \mathbb{Z}^n\neq\emptyset\Longleftrightarrow \exists\;\; \tilde{j}\in \Bbb N\;\; s.t. \;\;(\tilde{M}^{*\tilde{j}}\mathcal{Z}_{\tilde{D}}^n)\cap \mathbb{Z}^n\neq\emptyset.$$

Necessity: Suppose there exists an integer $j\in \Bbb N$ such that
$(M^{*j}\mathcal{Z}_{D}^n)\cap \mathbb{Z}^n\neq\emptyset$. By \eqref{eq(1.6.0)} and Proposition \ref{lem(4.3)}(i), we have
\begin{align*}
{M^*}^{j}\mathcal{Z}_D^n\subset{M^*}^{j}\mathcal{Z}(m_D)\subset \frac{1}{(\det(AB))^{j+1}}A^*({\tilde{M^*}}^{j}\mathcal {Z}(m_{\tilde{D}})+\Bbb Z^n).
\end{align*}
Thus $\frac{1}{(\det(AB))^{j+1}}A^*({\tilde{M^*}}^{j}\mathcal {Z}(m_{\tilde{D}})+\Bbb Z^n)\cap \Bbb Z^n\neq \emptyset$, which is equivalent to
$$E:=({\tilde{M^*}}^{j}\mathcal {Z}(m_{\tilde{D}})+\Bbb Z^n)\cap (\det(AB))^{j+1}A^{*-1}\Bbb Z^n\neq \emptyset.
$$
Combining this with $\det(A)A^{*-1}\in M_n(\Bbb Z)$, and using \eqref{eq(1.6.0)} again, we have
$$
({\tilde{M^*}}^{j}\mathcal{Z}_{\tilde{D}}^n+\mathbb{Z}^n)\cap\Bbb Z^n=({\tilde{M^*}}^{j}\mathcal {Z}(m_{\tilde{D}})+\Bbb Z^n)\cap \Bbb Z^n \supset E\neq\emptyset.
$$
Therefore, $({\tilde{M^*}}^{j}\mathcal{Z}_{\tilde{D}}^n)\cap \mathbb{Z}^n\neq\emptyset.$

Sufficiency:  Suppose there exist $\tilde{j}\in\mathbb{N}$ and $\tilde{\lambda}\in \mathcal{Z}_{\tilde{D}}^n$  such that
${\tilde{M^*}}^{\tilde{j}}\tilde{\lambda}\in \mathbb{Z}^n$. Since $AB=I \pmod p $ and $\mathcal{Z}_{\tilde{D}}^n\subset \mathring{E}_p^n$, it follows that $\det(AB)\tilde{\lambda}\in\mathcal{Z}_{\tilde{D}}^n +\Bbb Z^n=\mathcal{Z}(m_{\tilde{D}})$.
Let $\lambda:=B^{*-1}\det(AB)\tilde{\lambda}$, then  $\lambda \in B^{*-1}\mathcal{Z}(m_{\tilde{D}})= \mathcal{Z}(m_D)$  by \eqref{eq(3.11.1)}. It follows from \eqref{eq(2.10)} that
$$
{M^*}^{\tilde{j}}\lambda=
(A^*\tilde{M^*}^{\tilde{j}}B^*-pM_{\tilde{j}})B^{*-1}\det(AB)\tilde{\lambda}=
\det(AB)A^*(\tilde{M^*}^{\tilde{j}}\tilde{\lambda})-pM_{\tilde{j}}B^{*-1}\det(AB)\tilde{\lambda}
$$
for some $M_{\tilde{j}}\in M_n(\Bbb Z)$. Note that $\tilde{M^*}^{\tilde{j}}\tilde{\lambda},\;p\tilde{\lambda}\in \Bbb Z^n$ and $\det(B) {B^*}^{-1}\in M_n(\mathbb{Z})$, so that ${M^*}^{\tilde{j}}\lambda\in \mathbb{Z}^n.$ Hence $({M^*}^{\tilde{j}}\mathcal{Z}_{D}^n)\cap \mathbb{Z}^n\neq\emptyset$ follows from \eqref{eq(1.6.0)}.

Now we prove $n^*(\mu_{M,D})=n^*(\mu_{\tilde{M},\tilde{D}})$ if  $\det (M)\notin p\mathbb{Z}$.
Firstly, we show that $n^*(\mu_{M,D})\leq n^*(\mu_{\tilde{M},\tilde{D}})$. Suppose that $E(\Lambda)$ is an orthogonal set of
exponential functions in $L^2(\mu_{M,D})$,  then \eqref{1.4} implies that
$$
\left(\Lambda-\Lambda\right)\setminus\{0\}\subset\mathcal{Z}(\hat{\mu}_{M, D}).
$$
From Proposition \ref{lem(4.3)}(ii), it is easy to see  that
\begin{align*}
c_1B^*\left(\Lambda-\Lambda\right)\setminus\{0\}\subset c_1B^*\mathcal{Z}(\hat{\mu}_{M, D})\subset\mathcal{Z}(\hat{\mu}_{\tilde{M}, \tilde{D}}),
\end{align*}
where $c_1=\det(AB){\left|\det \tilde{(M)} \right|}^{(p-1)(p^n-1)}$.
This together with \eqref{1.4} yields that $E(c_1B^*\Lambda)$ is an orthogonal set of exponential functions in $L^2(\mu_{\tilde{M},\tilde{D}})$, and hence $n^*(\mu_{M,D})\leq n^*(\mu_{\tilde{M},\tilde{D}})$.

Secondly, we show that $n^*(\mu_{M,D})\geq n^*(\mu_{\tilde{M},\tilde{D}})$.  If $E(\tilde{\Lambda})$ is an orthogonal set of
exponential functions in $L^2(\mu_{\tilde{M},\tilde{D}})$,  then \eqref{1.4} implies that
$$
 (\tilde{\Lambda}-\tilde{\Lambda}) \setminus\{0\}\subset\mathcal{Z}(\hat{\mu}_{\tilde{M}, \tilde{D}}).
$$
In view of Proposition \ref{lem(4.3)}(ii), we have
\begin{align*}
c_2A^*(\tilde{\Lambda}-\tilde{\Lambda}) \setminus\{0\}\subset\mathcal{Z}(\hat{\mu}_{M, D}),
\end{align*}
where $c_2={\left|\det M \right|}^{(p-1)(p^n-1)}$.  Together with \eqref{1.4}, this yields that $E(c_2A^*\tilde{\Lambda})$ is an orthogonal set of exponential functions in $L^2(\mu_{M,D})$, and hence $n^*(\mu_{M,D})\geq n^*(\mu_{\tilde{M},\tilde{D}})$.
Combining with $n^*(\mu_{M,D})\leq n^*(\mu_{\tilde{M},\tilde{D}})$, we obtain $n^*(\mu_{M,D})= n^*(\mu_{\tilde{M},\tilde{D}})$.

Finally,  we prove $n^*(\mu_{\tilde{M},\tilde{D}})\leq p^n$. It is easy to see  that
$\gcd(\det(\tilde{M}),p)=1$ since $A,B\in GL_n(p)$ and $\det(M)\notin p\mathbb{Z}$. By Lemma \ref{lem(2.5)}(i), we have $\tilde{M^*}^{\tilde{j}}\mathring{E}_p^n=\mathring{E}_p^n \pmod {\mathbb{Z}}$ for any $\tilde{j}\geq 1$.
Thus it follows from $\mathcal{Z}_{\tilde{D}}^n\subset \mathring{E}_p^n$, \eqref{eq(1.6.0)} and \eqref{1.3} that
\begin{align}\label{eq(3.3.5)}
\mathcal{Z}(\hat{\mu}_{\tilde{M},\tilde{D}})
=\bigcup_{\tilde{j}=1}^\infty\tilde{M^*}^{\tilde{j}}\mathcal{Z}(m_{\tilde{D}})
\subset\bigcup_{\tilde{j}=1}^\infty\tilde{M^*}^{\tilde{j}}(\mathring{E}_p^n+\mathbb{Z}^n)
\subset \mathring{E}_p^n+\mathbb{Z}^n.
\end{align}
Now we suppose that  there exists $\tilde{\Lambda}=\{0,\tilde{\lambda}_1,\tilde{\lambda}_2,\cdots, \tilde{\lambda}_{p^n}\}$ such that $E(\tilde{\Lambda})$ is an orthogonal set of exponential functions in $L^2(\mu_{\tilde{M},\tilde{D}})$,  then the orthogonality of $E(\tilde{\Lambda})$ and \eqref{eq(3.3.5)}  imply  that $\tilde{\Lambda}\setminus\{0\}\subset \mathring{E}_p^n+\mathbb{Z}^n$. It is clear that $|\mathring{E}_p^n|=p^n-1$, by \eqref{eq(3.3.5)} again, there must exist $1\leq j_1,j_2\leq p^n$ and an integer vector $v'\in \mathbb{Z}^n$ such that $\tilde{\lambda}_{j_1}-\tilde{\lambda}_{j_2}=v' \notin \mathcal{Z}(\hat{\mu}_{\tilde{M},\tilde{D}})$,  which  contradicts  the  orthogonality of $E(\tilde{\Lambda})$. Therefore, $n^*(\mu_{\tilde{M},\tilde{D}})\leq p^n$.

The proof of Theorem \ref{th(3.2)} is completed.
\end{proof}

\begin{re}\label{re(3.3)}
 We point out that we do not consider the case where $\det(M)\in p\mathbb{Z}$ and $L^2(\mu_{M,D})$ has only finite orthogonal exponents in Theorem \ref{th(3.2)}. In this case,  it is difficult to verify whether $n^*(\mu_{M,D})=n^*(\mu_{\tilde{M},\tilde{D}})$ holds for the conjugated pair $(M,D)$ and $(\tilde{M}, \tilde{D})$ in $GL_n(p)$.
\end{re}

 In view of Theorem \ref{th(3.2)},  we can easily extend the results of Theorem \ref{th(1.6)}(ii) to the following general form.

\begin{cor}\label{th(3.3)}
Let $M\in M_2(\mathbb{Z})$ be an  expanding matrix and $D=\{(0,0)^t,(\alpha_1,\alpha_2)^t,(\beta_1,\beta_2)^t\}$  be an integer digit set with $\alpha_1\beta_2-\alpha_2\beta_1\notin 3\mathbb{Z}$. If $\det(M)\notin 3\mathbb{Z}$, then $L^2(\mu_{M,D})$ contains at most $9$ mutually
orthogonal exponential functions, and the number $9$ is the best. Moreover, the best number can be attained if and only if $AMB \in \mathfrak{M}_2  \pmod 3 $, where   $B=\left[ {\begin{array}{*{20}{c}}
\alpha_1&\beta_1\\
\alpha_2&\beta_2
\end{array}} \right]$  and $A\in GL_2(3)$  satisfy $AB=I \pmod 3$.
\end{cor}

\begin{proof}
Let $\tilde{M}=AMB$, $\tilde{D}=\mathcal {D}=\{(0,0)^t,(1,0)^t,(0,1)^t\}$, then $D=B\tilde{D}$. By Theorem \ref{th(3.2)}, we only need to prove  that  the conclusions hold for the pair $(\tilde{M},\tilde{D})$. Obviously, Theorem \ref{th(1.6)}(ii) implies that Corollary \ref{th(3.3)} holds since $\det(\tilde{M})=\det(AMB)\notin 3\mathbb{Z}$. The proof is completed.
\end{proof}

 Another well-known  planar digit set is $D_0=\{(0, 0)^t, (1, 0)^t, (0, 1)^t, (-1, -1)^t\}$, it is easy to calculate that $\mathcal{Z}_{D_0}^2=\{(1/2, 0)^t, (0, 1/2)^t, (1/2, 1/2)^t\}$. By Theorem 1.1 of \cite{Liu-Dong-Li_2017},  there exist at most $4$ mutually orthogonal exponential functions in  $L^2(\mu_{M,D_0})$ if $\det(M)\notin 2\mathbb{Z}$,  and the number 4 is the best.  Applying Theorem \ref{th(3.2)},  we can extend the above  digit set  to the following more general form
\begin{equation}\label{eq(3.14)}
D=\left\{\begin{pmatrix}
0\\
0\end{pmatrix},\begin{pmatrix}
\alpha_{1}\\
\alpha_{2}\end{pmatrix},\begin{pmatrix}
\beta_{1}\\
\beta_{2}\end{pmatrix},\begin{pmatrix}
-\alpha_{1}-\beta_{1}\\
-\alpha_{2}-\beta_{2}\end{pmatrix}\right\}\subset \Bbb Z^2\;\;\mbox{with}\;\; \alpha_1\beta_2-\alpha_2\beta_1\notin 2\mathbb{Z}.
\end{equation}

\begin{cor} \label{th(3.4)}
Let $M\in M_2(\Bbb Z)$ be an expanding integer matrix with $\det(M)\notin 2\mathbb{Z}$, and let $D$, $\mu_{M,D}$
be defined by \eqref{eq(3.14)} and \eqref{eq(1.1)}, respectively. Then there exist at most $4$ mutually
orthogonal exponential functions  in $L^2(\mu_{M,D})$, and the number $4$ is the best.
\end{cor}

\section{\bf The spectrality of generalized  Sierpinski measures \label{sect.4}}

In this section, we will consider the spectrality of generalized Sierpinski measure, which is  generated by an expanding matrix $M\in M_2(\mathbb{Z})$ and a digit set $D=\{(0,0)^t,(\alpha_1,\alpha_2)^t,(\beta_1,\beta_2)^t\}
$ $\subset\Bbb Z^2$ with $\alpha_1\beta_2-\alpha_2\beta_1\notin 3\mathbb{Z}$. The main  purpose  of this section is to prove
Theorem \ref{th(1.8)}. Firstly, we  give some  preparatory lemmas and theorems.

In the study of spectrality of self-affine measure $\mu_{M,D}$, it is very difficult to show that $\mu_{M,D}$ is a non-spectral measure when $L^2(\mu_{M,D})$ has an infinite orthogonal set.  Dai, He and Lau \cite{Dai-He-Lau_2014} got the following useful non-spectral criterion, which is widely used by many researchers.
\begin{lemma}\label{lem4.4} \cite{Dai-He-Lau_2014}
Let $\mu=\mu_1\ast\mu_2$ be the convolution of two probability measures $\mu_i$, $i=1, 2$,
and they are not Dirac measures.
If $E(\Lambda)$ is an orthonormal set in $L^2(\mu_1)$,
then $E(\Lambda)$ is also an orthonormal set in $L^2(\mu)$,
but $\Lambda$ cannot be a spectrum of $\mu$.
\end{lemma}

For the proof of Theorem \ref{th(1.8)}, the main  difficulty is to prove the situation ``$(i)\Rightarrow (iii)$''. In order to get it, we give a
non-spectral criterion for some special self-affine measures.
\begin{thm}\label{th(4.3.1)}
Let $M\in M_n(\mathbb{R})$ be an expanding matrix and $D$ be a finite  digit set of $\mathbb{R}^n$. If there exist  a real  $L>0$ and  an integer $j_0\geq 2$ such that $\emptyset\neq L(\mathcal{Z}(m_D)-\mathcal {Z}(m_D))\setminus\mathbb{Z}^n\subset L\mathcal {Z}(m_D)$,
\begin{equation} \label{eq(4.3.1)}
\left(L\bigcup_{j=1}^{j_0-1}{M^*}^j(\mathcal{Z}(m_D))\right)\bigcap\mathbb{Z}^n=\emptyset \ \ \ {\rm and}
 \ \ \ L{M^*}^{j}(\mathcal{Z}(m_D))\subset \mathbb{Z}^n \ {\rm for \  all } \  j\geq j_0,
\end{equation}
then $\mu_{M,D}$, defined as in \eqref{eq(1.1)}, is a non-spectral measure.
\end{thm}

\begin{proof}
Let $\mathcal {Z}_j:={M^*}^j\mathcal {Z}(m_D)$, and we denote
\begin{equation}\label{4.4}
\mu_1=\delta_{M^{-1}D}*\cdots*\delta_{M^{-{(j_0-1)}}D}*\delta_{M^{-(j_0+1)}D}*\delta_{M^{-(j_0+2)}D}*\cdots,  \ \ \ \
\mu_{2}=\delta_{M^{-j_0}D}*\delta_{M^{-(j_0+1)}D}*\cdots.
\end{equation}
Then $\mu_{M,D}=\mu_1*\delta_{M^{-{j_0}}D}=\mu_{2}*\delta_{M^{-1}D}*\cdots*\delta_{M^{-{(j_0-1)}}D}$. Moreover,
\begin{equation}\label{4.5}
\mathcal{Z}(\hat{\mu}_1)=\bigcup_{j\geq 1, j\neq j_0}\mathcal {Z}_ j\;\;\;\; \mbox{and} \; \;\;\;
\mathcal{Z}(\hat{\mu}_2)=\bigcup_{j\geq j_0}\mathcal {Z}_j.
\end{equation}

In order to obtain  that  $\mu_{M,D}$ is a non-spectral measure, from Lemma \ref{lem4.4}, it is sufficient to prove that any orthogonal set $E(\Lambda)$ of $L^2(\mu_{M,D})$ with $0\in \Lambda$ is either an orthogonal set of $L^2(\mu_1)$ or $L^2(\mu_2)$, i.e,
\begin{equation}\label{4.3.1}
(\Lambda-\Lambda)\setminus\{0\} \subset \mathcal{Z}(\hat{\mu}_1)\;\;\;\;\mbox{or}\;\;\;\;(\Lambda-\Lambda)\setminus\{0\} \subset \mathcal{Z}(\hat{\mu}_2).
\end{equation}
We first show that $\Lambda \subset \mathcal{Z}(\hat{\mu}_1)$ or $\Lambda \subset \mathcal{Z}(\hat{\mu}_2)$. Otherwise, there exist  different $\lambda_1,\lambda_2 \in \Lambda$ satisfying $\lambda_1\in \mathcal {Z}_{j_0}$ and $\lambda_2\in \mathcal {Z}_{k_0}$ for some $1\leq k_0\leq j_0-1$.
By the orthogonality of $\Lambda$, we have
\begin{equation}\label{4.6}
L\lambda_1-L\lambda_2=L{M^*}^{j_0}\xi_1-L{M^*}^{k_0}\xi_2=L{M^*}^{j}\xi_3
\end{equation}
for some $\xi_1,\xi_2,\xi_3\in \mathcal {Z}(m_D)$ and $j\geq 1$.
 We will prove that \eqref{4.6} cannot hold by dividing into the following two cases.

Case I: $j\geq j_0$.  By \eqref{eq(4.3.1)}, we see that  $L{M^*}^{j}\xi_3, L{M^*}^{j_0}\xi_1 \in \mathbb{Z}^n$. However, $L{M^*}^{k_0}\xi_2\notin \mathbb{Z}^n$
since $1\leq k_0\leq j_0-1$, this shows that  \eqref{4.6} does not  hold.

Case II: $j\leq j_0-1$.   If $k_0\neq j$, without loss of generality, we can assume $j>k_0$. Multiplying both sides of \eqref{4.6} by ${M^*}^{j_0-j}$ (if $j<k_0$, then multiplying ${M^*}^{j_0-k_0}$) yields that
\begin{equation}\label{4.7}
L{M^*}^{2j_0-j}\xi_1-L{M^*}^{k_0+j_0-j}\xi_2=L{M^*}^{j_0}\xi_3.
\end{equation}
Similar to Case I,  \eqref{4.7} also does not hold  since $L{M^*}^{j_0}\xi_3, L{M^*}^{2j_0-j}\xi_1\in \mathbb{Z}^n$ and $L{M^*}^{k_0+j_0-j}\xi_2\notin \mathbb{Z}^n$.

If $k_0=j$,  then \eqref{4.6}  is equivalent to
\begin{equation}\label{4.8}
L{M^*}^{j_0}\xi_1=L{M^*}^{k_0}(\xi_2+\xi_3).
\end{equation}
If $L(\xi_2+\xi_3)\in\mathbb{Z}^n$, then \eqref{4.8} implies that $L{M^*}^{j_0-k_0}\xi_1=L(\xi_2+\xi_3)\in \mathbb{Z}^n$,  it contradicts the assumption of \eqref{eq(4.3.1)}.   If $L(\xi_2+\xi_3)\notin\mathbb{Z}^n$, by noting that $\mathcal {Z}(m_D)=-\mathcal {Z}(m_D)$ and
$L(\mathcal{Z}(m_D)-\mathcal {Z}(m_D))\setminus  \Bbb Z^n
 \subset L\mathcal {Z}(m_D)$, we have $\xi_{23}:=\xi_2+\xi_3=\xi_2-(-\xi_3)\in\mathcal {Z}(m_D)$. Then  \eqref{eq(4.3.1)}  and \eqref{4.8} show that $L{M^*}^{k_0} \xi_{23}=L{M^*}^{j_0}\xi_1\in \mathbb{Z}^n$,  which is also a  contradiction. This tells us that $\Lambda \subset \mathcal{Z}(\hat{\mu}_1)$ or $\Lambda \subset \mathcal{Z}(\hat{\mu}_2)$.

Finally, we prove that \eqref{4.3.1} holds.
If $\Lambda \subset \mathcal{Z}(\hat{\mu}_2)$,   it follows easily from \eqref{eq(4.3.1)} that
\begin{equation}\label{4.7.1}
(\Lambda-\Lambda)\setminus\{0\}\subset \bigcup_{j=j_0}^\infty {M^*}^j\mathcal {Z}(m_D)=\mathcal{Z}(\hat{\mu}_2).
\end{equation}
If  $\Lambda \subset \mathcal{Z}(\hat{\mu}_1)$,  we  suppose $(\Lambda-\Lambda)\setminus\{0\}\not\subset\mathcal{Z}(\hat{\mu}_1)= \bigcup_{j=1,j\neq j_0}^\infty {M^*}^j\mathcal {Z}(m_D)$, i.e., there exist $\lambda_1, \lambda_2 \in \Lambda$ so that
\begin{equation}\label{eq(4.14)}
\lambda_1-\lambda_2={M^*}^{j_1}\xi_1-{M^*}^{j_2}\xi_2={M^*}^{j_0}\xi_3,
\end{equation}
where $\xi_1,\xi_2,\xi_3\in \mathcal {Z}(m_D)$ and positive integers  $j_1,j_2 \neq j_0$.
Similar to the discussion in Case II of \eqref{4.6}, it is easy to see that \eqref{eq(4.14)}
is untenable.
Hence $(\Lambda-\Lambda)\setminus\{0\}\subset\mathcal{Z}(\hat{\mu}_1)$. Combining this with \eqref{4.7.1}, we conclude that \eqref{4.3.1} holds.

This completes the proof of Theorem \ref{th(4.3.1)}.
\end{proof}

In fact, it is easy to check that  the condition $M\in\mathfrak{M}_1 \pmod 3$ in Theorem \ref{th(1.6)}(i) is equivalent to $M^*(1,-1)^t\in 3\mathbb{Z}^2$. Hence, we can rewrite Theorem \ref{th(1.6)}(i) as follows:
\begin{equation}\label{eq(4.8.1)}
\mu_{M,\mathcal {D}} \  {\rm \  is \ a \ spectral \  measure} \  \Longleftrightarrow (M, \mathcal {D}) {\rm  \ is \  admissible} \  \Longleftrightarrow
 M^*\left( {\begin{array}{*{20}{c}}
1\\
-1
\end{array}} \right)\in 3\mathbb{Z}^2.
\end{equation}

\begin{thm}\label{th(4.4)}
Let $M\in M_2(\mathbb{Z})$ be an expanding matrix and  $D=\{(0,0)^t,(\alpha_1,\alpha_2)^t,(\beta_1,\beta_2)^t\}$ be an integer digit set with $\alpha_1\beta_2-\alpha_2\beta_1\notin 3\mathbb{Z}$, and let $\mu_{M,D}$ be defined by \eqref{eq(1.1)}. Then the following statements  are equivalent:

(i) \ \  $\mu_{M,D}$ is a spectral measure;

(ii) \ \ $(M, D)$ is admissible;

(iii) \ \  $(AMB)^*(1,-1)^t\in 3\mathbb{Z}^2$,
where $B=\left[ {\begin{array}{*{20}{c}}
\alpha_1&\beta_1\\
\alpha_2&\beta_2
\end{array}} \right]$ and $A\in GL_2(3)$ satisfies $AB=I\pmod 3$.
\end{thm}
\begin{proof}
We will prove this theorem by the circle $(iii)\Rightarrow (ii)\Rightarrow (i)\Rightarrow (iii)$.

``$(iii)\Rightarrow (ii)$" \ \  Suppose $(AMB)^*(1,-1)^t\in 3\mathbb{Z}^2$. Let $ \tilde{M}=AMB$ and
$$
\tilde{D}=\mathcal {D}=\left\{\left( {\begin{array}{*{20}{c}}
	0\\
	0
	\end{array}} \right), \left( {\begin{array}{*{20}{c}}
	1\\
	0
	\end{array}} \right),\left( {\begin{array}{*{20}{c}}
	0\\
	1
	\end{array}} \right)\right\}.
$$ Then $D=B\tilde{D}$, $(M,D)$ and $(\tilde{M}, \tilde{D})$ are conjugate in $GL_n(p)$ (through the matrix $A,B\in GL_n(p)$). Obviously, $(\tilde{M}, \tilde{D})$ is admissible by Theorem \ref{th(1.6)}(i) and \eqref{eq(4.8.1)}, hence Theorem \ref{th(3.1)} implies  that $(M, D)$ is admissible.

``$(ii)\Rightarrow (i)$" \ \  The assertion follows directly from Theorem \ref{th(1.3)}.

``$(i)\Rightarrow (iii)$"  \ \ It is equivalent to show  that  $\mu_{M,D}$ is a non-spectral measure if $(AMB)^* (1,-1)^t$ $\notin 3\mathbb{Z}^2.$ We will prove this by dividing into two cases:
Case I:  ${(AMB)^*}^j (1,-1)^t$ $\notin 3\mathbb{Z}^2$ for all $j\geq 1$; Case II: there exists $j_0\geq 2$ such that  ${(AMB)^*}^j (1,-1)^t \in 3\mathbb{Z}^2$ for all $j\geq j_0$ and ${(AMB)^*}^j (1,-1)^t \notin 3\mathbb{Z}^2$ for all $j< j_0$.

Case I: By noting that $(1,-1)^t=-(-1,1)^t$ and ${(AMB)^*}^j(1,-1)^t\notin 3\mathbb{Z}^2$ for all $j\geq 1$,  we have
\begin{equation} \label{eq(4.3)}
{\tilde{M^*}}^j\left\{\left( {\begin{array}{*{20}{c}}
\frac{1}{3}\\
-\frac{1}{3}
\end{array}} \right) ,\left( {\begin{array}{*{20}{c}}
-\frac{1}{3}\\
\frac{1}{3}
\end{array}} \right)\right\}={(AMB)^*}^j\left\{\left( {\begin{array}{*{20}{c}}
\frac{1}{3}\\
-\frac{1}{3}
\end{array}} \right) ,\left( {\begin{array}{*{20}{c}}
-\frac{1}{3}\\
\frac{1}{3}
\end{array}} \right)\right\}\subset \mathring{E}_3^2  \pmod {\mathbb{Z}^2} ,\;\;\;j\geq 1.
\end{equation}
On the other hand, it follows from \eqref{eq(3.11.1)} that
\begin{equation}\label{eq(4.4)}
B^* \mathcal {Z}(m_D)=\mathcal {Z}(m_\mathcal {D})=
\left\{\left( {\begin{array}{*{20}{c}}
\frac{1}{3}\\
-\frac{1}{3}
\end{array}} \right) ,\left( {\begin{array}{*{20}{c}}
-\frac{1}{3}\\
\frac{1}{3}
\end{array}} \right)\right\}+\mathbb{Z}^2.
\end{equation}
Now we will show that there exist at most $9$   mutually
orthogonal exponential functions in $L^2(\mu_{M,D})$. Suppose on the contrary that there exists
$\Lambda=\{0\}\cup \{\lambda_i\}^9_{i=1}$ such that $E(\Lambda)$ is an orthogonal set of $L^2(\mu_{M, D})$. It follows from Proposition \ref{lem(4.3)}(i) that
\begin{align}\label{eq:4.19}
{M^*}^{j}\mathcal{Z}(m_D)\subset\frac{1}{(3l+1)^{j+1}}A^*\left({\tilde{M^*}}^{j}\mathcal {Z}(m_{\mathcal {D}})+\Bbb Z^2\right) \quad\mbox{for}\ j\geq 1,
\end{align}
where $\det(AB):=3l+1$ for some integer $l$. Thus  by \eqref{1.3} and \eqref{1.4}, we can let
\begin{align*}
\lambda_i=\frac{1}{(3l+1)^{j_i+1}}A^*\left({\tilde{M}}^{*j_i}\tilde{\lambda}_i+\nu_i\right),
\end{align*}
where $\tilde{\lambda}_i\in \mathcal {Z}(m_\mathcal {D})$, $\nu_i\in\Bbb Z^2$ and $j_i\geq 1$. Applying \eqref{eq(4.3)} and \eqref{eq(4.4)}, we have ${\tilde{M}}^{j_i}\tilde{\lambda}_i\in \mathring{E}_3^2  \pmod {\mathbb{Z}^2} $ for all $1\leq i\leq 9$. Hence there must exist $1\leq i_1, i_2\leq 9$ such that ${\tilde{M}}^{j_{i_1}}\tilde{\lambda}_{i_1}={\tilde{M}}^{j_{i_2}}\tilde{\lambda}_{i_2}\in \mathring{E}_3^2  \pmod {\mathbb{Z}^2} $.
The orthogonality of $\lambda_{i_1}$ and $\lambda_{i_2}$
implies  that  there exist $\tilde{\lambda}_{i_{12}}\in \mathcal {Z}(m_{\mathcal {D}})$ and an integer $j_{i_{12}}$ such that
\begin{align*}
\lambda_{i_1}-\lambda_{i_2}&=\frac{1}{(3l+1)^{j_{i_1}+1}}A^*\left({\tilde{M}}^{*j_{i_1}}\tilde{\lambda}_{i_1}+\nu_{i_1}\right)-
\frac{1}{(3l+1)^{j_{i_2}+1}}A^*\left({\tilde{M}}^{*j_{i_2}}\tilde{\lambda}_{i_2}+\nu_{i_2}\right)\\
&=\frac{1}{(3l+1)^{j_{i_{12}}+1}}A^*\left({\tilde{M}}^{*j_{i_{12}}}\tilde{\lambda}_{i_{12}}+\nu_{i_{12}}\right).
\end{align*}
Multiplying both sides of the above equation by $(3l+1)^{j_{i_1}+j_{i_2}+j_{i_{12}}+1}{A^*}^{-1}$, we have
\begin{align}\label{eq(4.6)}
&\quad(3l+1)^{j_{i_2}+j_{i_{12}}}\left({\tilde{M}}^{*j_{i_1}}\tilde{\lambda}_{i_1}+\nu_{i_1}\right)-
(3l+1)^{j_{i_1}+j_{i_{12}}}\left({\tilde{M}}^{*j_{i_2}}\tilde{\lambda}_{i_2}+\nu_{i_2}\right)\nonumber\\
&=(3l+1)^{j_{i_1}+j_{i_2}}\left({\tilde{M}}^{*j_{i_{12}}}\tilde{\lambda}_{i_{12}}+\nu_{i_{12}}\right).
\end{align}
It is easy to see that the left side of \eqref{eq(4.6)} is an integer  vector of $\mathbb{Z}^2$  since ${\tilde{M}}^{*j_{i_1}}\tilde{\lambda}_{i_1}={\tilde{M}}^{*j_{i_2}}\tilde{\lambda}_{i_2}\in \mathring{E}_3^2  \pmod {\mathbb{Z}^2}$. However,  ${\tilde{M}}^{*j}\mathcal {Z}(m_{\mathcal {D}})\subset \mathring{E}_3^2  \pmod {\mathbb{Z}^2}$  for all $j\geq 1$ implies that the right side of \eqref{eq(4.6)} cannot be an integer vector of $\mathbb{Z}^2$.
This contradiction shows that there are at most $9$   mutually
orthogonal exponential functions in $L^2(\mu_{M,D})$.

We now prove that  $\mu_{M,D}$ is a  non-spectral measure if $M$ belongs to Case II, i.e., there exists $j_0\geq 2$ such that
${\tilde{M}}^{*j} (1,-1)^t \in 3\mathbb{Z}^2$ for all $j\geq j_0$ and ${\tilde{M}}^{*j} (1,-1)^t \notin 3\mathbb{Z}^2$ for all $j< j_0$. Hence, \eqref{eq(4.4)} gives that
\begin{align}\label{eq?4.21.0}
{\tilde{M}}^{*j}\mathcal {Z}(m_\mathcal {D})\subset \Bbb Z^2\;\;\mbox{for}\;j\geq j_0\;\mbox{and}\;\;{\tilde{M}}^{*j}\mathcal {Z}(m_\mathcal {D}) \subset\mathring{E}_3^2+ {\mathbb{Z}^2} \;\;\mbox{for}\;j< j_0.
\end{align}
Let $L:=(3l+1)^{j_0+1}$, it follows from \eqref{eq:4.19}  that
\begin{align*}
L\bigcup_{i=1}^{j_0-1}{M^*}^i\mathcal {Z}(m_D)\subset \bigcup_{i=1}^{j_0-1}(3l+1)^{j_0-i}A^*\left({\tilde{M}}^{*i}\mathcal {Z}(m_\mathcal {D})+\Bbb Z^2\right)
\end{align*}
and
\begin{align*}
L{M^*}^j\mathcal {Z}(m_D)&=L{M^*}^{j-j_0}{M^*}^{j_0}\mathcal {Z}(m_D)\subset L{M^*}^{j-j_0} \frac{1}{(3l+1)^{j_0+1}}A^*\left({\tilde{M}}^{*j_0}\mathcal {Z}(m_\mathcal {D})+\Bbb Z^n\right)\\
&= {M^*}^{j-j_0}A^*\left({\tilde{M}}^{*j_0}\mathcal {Z}(m_\mathcal {D})+\Bbb Z^n\right)\;\;\mbox{for}\;\;j\geq j_0.
\end{align*}
Combining this with \eqref{eq?4.21.0} and $A\in GL_2(3)$,  we conclude that
\begin{align}\label{eq:4.23}
L\bigcup_{i=1}^{j_0-1}{M^*}^i\mathcal {Z}(m_D)\cap \Bbb Z^2=\emptyset\;\;\mbox{and}\;\;L{M^*}^j\mathcal {Z}(m_D)\subset \Bbb Z^2\;\; \mbox{for}\;\; j\geq j_0.
\end{align}
 Also, it follows from $3l+1=\det(AB)$ and \eqref{eq(4.4)} that $$
L\mathcal {Z}(m_D)=(3l+1)^{j_0+1}\mathcal {Z}(m_D)=(3l+1)^{j_0}\det(AB){B^*}^{-1}\mathcal {Z}(m_\mathcal {D})=\mathscr{M}\mathcal {Z}(m_\mathcal {D}),
$$
where $\mathscr{M}:=(3l+1)^{j_0}\det(AB){B^*}^{-1}$. Since $\mathscr{M}$ is an integer matrix and $(\mathcal {Z}(m_\mathcal {D})-\mathcal {Z}(m_\mathcal {D}))\setminus \mathbb{Z}^n =\mathcal {Z}(m_\mathcal {D})$, we obtain that
\begin{align}\label{eq(4.15.1)} \nonumber
L(\mathcal {Z}(m_D)-\mathcal {Z}(m_D))\setminus \mathbb{Z}^n
&=(\mathscr{M}\mathcal {Z}(m_\mathcal {D})-\mathscr{M}\mathcal {Z}(m_\mathcal {D}))\setminus \mathbb{Z}^n  \\ \nonumber
&\subset \mathscr{M}\left((\mathcal {Z}(m_\mathcal {D})-\mathcal {Z}(m_\mathcal {D}))\setminus\mathbb{Z}^n  \right)\\
&=\mathscr{M}\mathcal {Z}(m_\mathcal {D})=L\mathcal {Z}(m_D).
\end{align}
Together with  \eqref{eq:4.23}, it yields that the pair $(M,D)$ satisfies the conditions of Theorem \ref{th(4.3.1)}. Hence $\mu_{M.D}$ is a non-spectral measure.

This  completes  the proof of Theorem \ref{th(4.4)}.
\end{proof}

Now we will give an example to indicate that Theorem \ref{th(4.4)} may not hold for $\det(B)\in 3\mathbb{Z}$.
Let the integer matrix $M$ and the integer digit set $D$ be defined as follows:
\begin{equation}\label{4.15}
M=\begin{bmatrix}
0 & 3b+1\\
3a& 0 \end{bmatrix}\quad  {\rm {and}}\quad
D=\left\{\left( {\begin{array}{*{20}{c}}
0\\
0
\end{array}} \right), \left( {\begin{array}{*{20}{c}}
1\\
0
\end{array}} \right),\left( {\begin{array}{*{20}{c}}
2\\
9
\end{array}} \right)\right\},
\end{equation}
 where $a,b\geq 3$.
By a direct calculation, we have
\begin{equation}\label{4.16}
\mathcal{Z}(m_D)=\mathcal{Z}_1\cup \mathcal{Z}_2,
\end{equation}
where
\begin{gather*}
\mathcal{ Z}_1=\left\{
\begin{pmatrix}
\frac{1}{3}+k_1 \\ \frac{\ell_1}{9}+k_2
\end{pmatrix}
:0\leq\ell_1\leq8,k_1,k_2\in\mathbb{Z}
\right\}\;\;\mbox{and}\;\;
\mathcal{ Z}_2=\left\{
\begin{pmatrix}
\frac{2}{3}+k_1^\prime \\ \frac{\ell_2}{9}+k_2^\prime
\end{pmatrix}
:0\leq\ell_2\leq8,k_1^\prime,k_2^\prime\in\mathbb{Z}
\right\}.
\end{gather*}
It is easy to check that $M^*\mathcal{Z}(m_D) \cap \mathbb{Z}^2=\emptyset$, which implies that $(M, D)$ is not admissible. However, we can show that $\mu_{M,D}$ is a spectral measure.

\begin{ex} \label{th(4.7)}
Let $\mu_{M,D}$ be defined by \eqref{eq(1.1)}, where $M,D$ are  given in \eqref{4.15}. Then $\mu_{M,D}$ is a spectral measure.
\end{ex}

\begin{proof}
Let $C=\{(0, 0)^t, (\frac{1}{3}, 0)^t, (\frac{2}{3}, 0)^t\}$, and denote
\begin{equation}\label{eq:def_Lmd} \Lambda_n=\sum_{i=1}^{n} {M^*}^iC \quad{\rm and} \quad \Lambda=\bigcup_{n=1}^{\infty} \Lambda_n.
\end{equation}
We will prove that
$\Lambda$  is a spectrum of $\mu_{M,D}$.

Firstly, we show that there  exist  $c>0$ and $\eta>0$ such that
\begin{align}\label{4.19}
|\hat{\mu}_{M, D}(\xi)| \geq c
\ {\rm for \ each} \  \xi \in T_{\eta}:=\{x: {\rm{dist}}(x, T)\leq \eta\},
\end{align}
where $T=T(M,C):=\sum_{i=1}^\infty M^{-i}C.$ A direct calculation shows that
\begin{equation}\label{4.18}
M^{-i}=\left\{ {\begin{array}{*{20}{c}}
\left(\frac{1}{3a(3b+1)}\right)^{\frac{i}{2}}\begin{bmatrix}
1 & 0\\
0& 1 \end{bmatrix} \ \ \ \ \ \ \ \ \ \ \ \ \ \ \ \ \ \   {\rm if } \  i \ {\rm is \  even},  \\
\left(\frac{1}{3a(3b+1)}\right)^{\frac{i-1}{2}}\begin{bmatrix}
0 & \frac{1}{3a}\\
\frac{1}{3b+1}& 0 \end{bmatrix}\ \ \ \ \ \ \ \ \ \ \ \   {\rm if } \  i \ {\rm is \  odd}.
\end{array}}
\right.
\end{equation}
By  $a ,b\geq3$ and \eqref{4.16}, we have $T\subset [0,\frac{2}{3}\sum_{i=1}^\infty\frac{1}{9}]^2=[0,\frac{1}{12}]^2$ and ${\rm{dist}}(T,\mathcal{Z}(m_D))>0$. Let $\eta:={\rm{dist}}(T,\mathcal{Z}(m_D))/2$, recall that $T_{\eta}=\{x: {\rm{dist}}(x, T)\leq \eta\}$, then ${\rm{dist}}(T_\eta,\mathcal{Z}(m_D))\geq \eta>0$.
Combining this with the continuous of the function $|m_D(x)|$, we have
\begin{align}\label{eq:4.23.1}
\beta:=\inf\{|m_D(x)|: x \in T_{\eta}\}> 0.
\end{align}

 Let $\xi\in T_{\eta}$,
then  by the definition of $T_{\eta}$, there exists $\xi_0\in T$ such that ${\rm{dist}}(\xi, \xi_0)\leq \eta$.
Noting that $0\in C$, by the definition of $T$,
we can obtain $M^{-i} T\subset T$ for each $i\geq 1$.  Consequently, $M^{-i}\xi_0\in  T$. Since $M$ is an expanding matrix,  it follows  that
$$
{\rm{dist}}(M^{-i}\xi, T(M, C) )\leq {\rm{dist}}(M^{-i}\xi, M^{-i}\xi_0)\leq{\rm{dist}(\xi, \xi_0)}\leq \eta.
$$
This proves $M^{-i}\xi\in T_{\eta}$. Then using \eqref{eq:4.23.1}, we get
\begin{align}\label{4.20}
|m_D (M^{-i}\xi)|\geq \beta  \ {\rm for \ all} \ i\in\mathbb{N} \ {\rm and} \   \xi\in T_{\eta}.
\end{align}
Because $T_{\eta}$ is a bounded set and $M$ is an expanding matrix, it is clear that $\lim\limits_{i\to\infty}\sup\{| M^{-i}\xi|:\xi\in T_\eta\}=0$. Combining this with the facts $|\hat{\mu}_{M, D}(0)|=1$ and $\hat{\mu}_{M, D}(\zeta)$ is continuous in $\mathbb{R}^2$, we can find an integer $i_0$ such that
$$|\hat{\mu}_{M, D}(M^{-i}\xi)|\geq \frac{1}{2}\;\; \mbox{for}\;\; \xi\in T_{\eta}\;\;\mbox{and}\;\; i\geq i_0.
$$
Thus it follows from \eqref{4.20} that
\begin{align*}
|\hat{\mu}_{M, D}(\xi)|&=\left|\prod_{i=1}^\infty m_D (M^{-i}\xi)\right|=\left|\prod_{i=1}^{i_0} m_D(M^{-i}\xi)\right|\left|\prod_{i=i_0+1}^\infty m_D(M^{-i}\xi)\right|\\
&=|\hat{\mu}_{M, D}(M^{-i_0}\xi)|\prod_{i=1}^{i_0} |m_D (M^{-i}\xi)|\geq\frac{ \beta^{i_0}}{2}:=c.
\end{align*}
Therefore, \eqref{4.19} holds.

In the following proof, we write
\begin{equation}\label{4.21}
\mu_n=\delta_{M^{-1}D}*\cdots*\delta_{M^{-n}D}  \ \ \ \ {\rm {and}}\  \ \ \
\mu_{>n}=\delta_{M^{-(n+1)}D}*\delta_{M^{-(n+2)}D}*\cdots.
\end{equation}
Then $\mu_{M,D}=\mu_n*\mu_{>n}$.

Secondly,  we claim that $\Lambda_n=\sum_{i=1}^{n} {M^*}^iC$  is  a  spectrum  of $\mu_n$. In fact, since the cardinality of $\Lambda_n$  is equivalent to the  dimension of $L^2(\mu_n)$, we only need to show that $E(\Lambda_n)$ is an orthogonal set of $L^2(\mu_n)$. Recall that
$C=\{(0,0)^t,(\frac{1}{3},0)^t, (\frac{2}{3},0)^t \}$. Thus for any  $\lambda,\; \lambda^\prime\in \Lambda_n$ with $\lambda\neq \lambda^\prime$, we have $\lambda=\sum_{i=1}^{n}M^{*i}\ell_i$ and $\lambda^\prime=\sum_{i=1}^{n}M^{*i}\ell_i^\prime$, where $\ell_i,\ell_i^\prime\in C$ for all $1\leq i\leq n$. Let $\kappa$ ($1\leq\kappa\leq n$) be the first index satisfying $\ell_{\kappa}\neq \ell_{\kappa}^\prime$, then
\begin{align}\label{4.23}
\lambda-\lambda^\prime&= {M^*}^{\kappa} \big((\ell_{\kappa}-\ell_{\kappa}^\prime)+
 M^* (\ell_{\kappa+1}-\ell_{\kappa+1}^\prime)+
\sum_{i=\kappa+2}^{n} {M^*}^{i-\kappa} (\ell_i-\ell_i^\prime)\big).
\end{align}
Noting that  $M=\begin{bmatrix}
0 & 3b+1\\
3a& 0 \end{bmatrix}$  and $\ell_{\kappa}\neq \ell_{\kappa}^\prime$, we have $\ell_{\kappa}-\ell_{\kappa}^\prime\in\mathcal{ Z}(m_D)$ (see \eqref{4.16}) and
$$ M^* (\ell_{\kappa+1}-\ell_{\kappa+1}^\prime)\in\left\{ {\left( {\begin{array}{*{20}{c}}
0\\
0
\end{array}} \right),\left( {\begin{array}{*{20}{c}}
0\\
\frac{1}{3}
\end{array}} \right),\left( {\begin{array}{*{20}{c}}
0\\
\frac{2}{3}
\end{array}} \right)} \right\}+\mathbb{Z}^2.$$
Thus $(\ell_{\kappa}-\ell_{\kappa}^\prime)+ M^* (\ell_{\kappa+1}-\ell_{\kappa+1}^\prime)\in\mathcal{ Z}(m_D)$ and ${M^*}^{i-\kappa} (\ell_{i}-\ell_{i}^\prime)\in\mathbb{Z}^2$ for $i\geq \kappa+2$.  Therefore,  we deduce from \eqref{4.23} and \eqref{4.16} that
\begin{align}\label{4.24}
\lambda-\lambda^\prime\in  {M^*}^\kappa (\mathcal{ Z}(m_D)+\mathbb{Z}^2)= {M^*}^\kappa \mathcal{Z}(m_D)\subset \bigcup_{i=1}^{n} {M^*}^i \mathcal{Z}(m_D)=\mathcal{ Z}(\hat{\mu}_n).
\end{align}
This implies that $E(\Lambda_n)$ is an orthogonal set of $L^2(\mu_n)$. Hence the claim follows.

Finally, we prove that $\Lambda=\bigcup_{n=1}^{\infty} \Lambda_n$ is a spectrum of $\mu_{M,D}$. For any $\lambda,\lambda^\prime\in\Lambda$ with $\lambda\neq \lambda^\prime$, there exists an integer $n$ such that  $\lambda,\lambda^\prime\in \Lambda_n$.  As $\mu_{M,D}=\mu_n*\mu_{>n}$, it follows that $\mathcal{Z}(\hat{\mu}_n)\subset\mathcal{ Z}(\hat{\mu}_{M,D})$. This together with \eqref{4.24} yields that $\lambda-\lambda^\prime\in\mathcal{ Z}(\hat{\mu}_{M,D})$. Hence, $E(\Lambda)$ is an orthogonal set of $L^2(\mu_{M,D})$.

Let
$$
Q_n(\xi)=\sum_{\lambda\in\Lambda_n}|\hat{\mu}_{n}(\xi+\lambda)|^2
\quad \text{and} \quad Q(\xi)=\sum_{\lambda\in\Lambda}|\hat{\mu}_{M, D}(\xi+\lambda)|^2.
$$
In order to show that $\Lambda$ is  a spectrum of $\mu_{M,D}$, we only need to prove $Q(\xi)\equiv 1$ for all $\xi\in \mathbb{R}^2$ by Lemma 3.3 of \cite{Jorgensen-Pedersen_1998}. Since $Q$ is an entire function, it is sufficient to determine the value of $Q(\xi)$ for $|\xi|\leq \eta={\rm{dist}}(T,\mathcal{Z}(m_D))/2$.  For any $n\in \mathbb{N}$,
by \eqref{1.2} and the definition of $\mu_n$,
we have the following equation:
\begin{align}\label{eq:Q2n}
Q_{2n}(\xi)&=Q_{n}(\xi)+\sum_{\lambda\in\Lambda_{2n}\setminus\Lambda_{n}}| \hat{\mu}_{M,D} (\xi+\lambda)|^2\nonumber\\
&=Q_{n}(\xi)+\sum_{\lambda\in\Lambda_{2n}\setminus\Lambda_{n}}| \hat{\mu}_{2n} (\xi+\lambda)|^2
| \hat{\mu}_{M,D} ({M^{-2n}}(\xi+\lambda))|^2.
\end{align}
Obviously,  $M^{-2n}\lambda\in T$ provided that $\lambda\in\Lambda_{2n}$.
This implies that ${\rm{dist}}({M^{-2n}}(\xi+\lambda),T)
\leq |{M^{-2n}}(\xi)| \leq \eta$ for $|\xi|\leq \eta$, and then ${M^{-2n}}(\xi+\lambda)\in T_\eta$.
By \eqref{4.19} and the fact $\Lambda_{2n}$ is a spectrum of $\mu_{2n}$, the equation \eqref{eq:Q2n} can be changed to
\begin{align}\label{4.25}
Q_{2n}(\xi)&\geq Q_{n}(\xi)+c^2\sum_{\lambda\in\Lambda_{2n}\setminus\Lambda_{n}}| \hat{\mu}_{2n}(\xi+\lambda)|^2 \nonumber \\
&=Q_{n}(\xi)+c^2(1-\sum_{\lambda\in\Lambda_{n}}|\hat{\mu}_{2n} (\xi+\lambda)|^2).
\end{align}
For any $\lambda\in\Lambda_{n}$ and $|\xi|\leq \eta<1$,
it follows from \eqref{eq:def_Lmd} and \eqref{4.18} that
\begin{align*}
|{M^{-2n}}(\xi+\lambda)|\leq 3^{-2n}
\left(1+\frac{2(3+3^2+\cdots +3^n)\sqrt{2}}{3}\right)
< 3^{-(n-1)}.
\end{align*}
Define $l_n=\min_{|\xi|\leq 3^{-n}}|\hat{\mu}_{M, D}(\xi)|$ for $n\in \mathbb{N}$.
It follows that $\lim_{n\rightarrow\infty}l_{n} =1$.
Hence we can obtain that
\begin{align}\label{4.26}
|\hat{\mu}_{M,D}(\xi+\lambda)|=|\hat{\mu}_{2n}(\xi+\lambda)||\mu_{M, D}
({M^{-2n}}(\xi+\lambda))|\geq l_{n-1}|\hat{\mu}_{2n}(\xi+\lambda)|.
\end{align}
According to \eqref{4.25} and \eqref{4.26}, we get
\begin{align*}
Q_{2n}(\xi)&\geq Q_{n}(\xi)+c^2(1-\sum_{\lambda\in\Lambda_{n}}|\hat{\mu}_{2n}(\xi+\lambda)|^2)\\
&\geq Q_{n}(\xi)+c^2(1-\frac{1}{l_{n-1}}\sum_{\lambda\in\Lambda_{n}}|\hat{\mu}_{M,D}(\xi+\lambda)|^2)\\
&= Q_{n}(\xi)+c^2(1-\frac{1}{l_{n-1}}Q_{n}(\xi)).
\end{align*}
Letting $n\rightarrow \infty$,
we can obtain that $Q(\xi)\geq Q(\xi)+c^2(1-Q(\xi))$.
Since $E(\Lambda)$ is an orthonormal set of $L^2(\mu_{M,D})$, by Lemma 3.3 of \cite{Jorgensen-Pedersen_1998}, we have $Q(\xi)\leq 1$ for every point $\xi\in\mathbb{R}^2$.
This implies that $Q(\xi)\equiv 1$ for $|\xi|\leq\eta$, and hence $\mu_{M,D}$ is a spectral measure by Lemma 3.3 of \cite{Jorgensen-Pedersen_1998} again.
\end{proof}

\end{document}